\documentclass[a4paper,12pt]{article}

\usepackage{amsmath, amsfonts, amssymb, graphicx, fullpage}
\usepackage{authblk}
\usepackage{subfigure}
\usepackage[numbers,sort&compress]{natbib}
\newtheorem{theorem}{Theorem}[section]
\newtheorem{corollary}[theorem]{Corollary}
\newtheorem{lemma}[theorem]{Lemma}

\newenvironment{proof}{{\em Proof.}}{\qquad$\Box$}
\newcommand\essinf{\textnormal{ess\,inf}}
\newcommand\esssup{\textnormal{ess\,sup}}
\newcommand{\asym}{\sim}

\newcommand\leb{\ensuremath{leb}}
\newcommand{\p}{\mathcal{P}}

\newcommand{\Z}{\mathbb{Z}}
\newcommand{\cone}{\mathcal{C}}
\newcommand\Deps{\ensuremath{\Delta^{\epsilon,-}}}
\newcommand\Depss{\ensuremath{\Delta^{\epsilon,+}}}
\newcommand\Depsj{\ensuremath{\Delta^{\epsilon,*}}}
\newcommand\Peps{{\Pi^\epsilon}}
\newcommand\Feps{{F^\epsilon}}
\newcommand\PFeps{{\p}_{\Feps}}
\newcommand\PTeps{\LTeps}
\newcommand\LTeps{{\mathcal{L}}_T^\epsilon}
\newcommand\LFeps{\mathcal{L}^\epsilon_\Feps}
\newcommand\ADeps{{\mathcal{A}_\Delta^\epsilon}}
\newcommand\Aeps{{\mathcal{A}^\epsilon}}
\newcommand\Leps{{\mathcal{L}^\epsilon}}
\newcommand\tLeps{{\tilde{\mathcal{L}}^\epsilon}}

\title{Spectral degeneracy and escape dynamics for intermittent maps with a hole}

\author[1]{Gary Froyland}
\author[2]{Rua Murray}
\author[1]{Ognjen Stancevic}
\affil[1]{School of Mathematics and Statistics, University of New South Wales, Sydney NSW 2052, Australia}
\affil[2]{Department of Mathematics and Statistics, University of Canterbury, Private Bag 4800, Christchurch 8140, New Zealand}

\begin{document}
\maketitle
\begin{abstract}
We study intermittent maps from the point of view of metastability. Small neighbourhoods of an intermittent fixed point and their complements form pairs of almost-invariant sets. Treating the small neighbourhood as a hole, we first show that the absolutely continuous conditional invariant measures (ACCIMs) converge to the ACIM as the length of the small neighbourhood shrinks to zero. We then quantify how the escape dynamics from these almost-invariant sets are connected with the second eigenfunctions of Perron-Frobenius (transfer) operators when a small perturbation is applied near the intermittent fixed point. In particular, we describe precisely the scaling of the second eigenvalue with the perturbation size, provide upper and lower bounds, and demonstrate $L^1$ convergence of the positive part of the second eigenfunction to the ACIM as the perturbation goes to zero. This perturbation and associated eigenvalue scalings and convergence results are all compatible with Ulam's method and provide a formal explanation for the numerical behaviour of Ulam's method in this
nonuniformly hyperbolic setting. The main results of the paper are illustrated with numerical computations.
\end{abstract}

\noindent{{\em Keywords:\/} Escape rate, Pomeau-Manneville maps, Perron-Frobenius operators, Ulam's method, Young towers}\newline

\noindent{{2010 MSC numbers:} 37E05, 37M25, 37D25}


\section{Introduction}

For maps $T:X\to X$ with good expansion properties, quantities such as the rate of decay of correlations and the rate of escape into a hole may be read off the spectrum of the corresponding Perron-Frobenius (or transfer) operator $\mathcal{P}:(\mathcal{B},\|\cdot\|)\circlearrowleft$ acting on a suitable Banach space $(\mathcal{B},\|\cdot\|)$ of regular observables.
The rate of decay of correlations is controlled by the gap in the spectrum of the Perron-Frobenius operator between the unit circle $\{z\in \mathbb{C}:|z|=1\}$ and the rest of the spectrum of $\mathcal{P}$.
Of particular interest for metastability is the situation where this gap is between the eigenvalue 1, with corresponding eigenfunction $\varphi$, the density of an absolutely continuous invariant measure (ACIM), and a second real eigenvalue $\lambda_2\in(0,1)$ with corresponding eigenfunction $\psi$ \citep{DJ99,DFS00}.

The notions of \emph{metastability} and \emph{almost-invariant sets} for closed systems are frequently explored via sub-unit eigenvalues of Perron-Frobenius operators, and have been used in a variety of real applications \citep{DS04,DJLM05,FPET07,FSM10}; also, extensions to random dynamical systems \citep{FLQ10,FLS10} have been developed recently.

If a small hole $H$ is excised from the domain $X$, as in early work of Pianigiani and Yorke~\citep{PY79}, one defines a conditional Perron-Frobenius operator $\mathcal{P}_{X\setminus H}(f):=\chi_{X\setminus H}\cdot\mathcal{P}(f\cdot\chi_{X\setminus H})$. If $T$ has good expansion properties, a \emph{conditional} ACIM (ACCIM) may exist \citep{PY79}, with density $\varphi_H$ satisfying $\mathcal{P}_{X\setminus H}\varphi_H=\lambda_H\varphi_H$, $0<\lambda_H\lesssim 1$. The density $\varphi_H$ describes the distribution of trajectories that have not yet fallen into the hole and the geometric rate at which trajectories of $T$ fall into $H$ is given by $\lambda_H$ \citep{PY79,DY06}.

In recent years there has been renewed interest in open systems and escape rates, stimulated in part by a series of papers by Collet~\emph{et al.}~\citep{CMS94, CMS96, CMS97, CMS99}  and, more recently, by a survey paper of Demers and Young \citep{DY06}. Much work has been carried out to show the existence of physically relevant conditionally invariant measures in various settings. First results \citep{PY79,CMS97} were obtained for maps that are Markov when the hole is removed. Shortly thereafter, similar results emerged for non-Markov expanding maps of the interval \citep{BC02, LiM03, Dem05b}, unimodal maps \citep{HY02, Dem05a}, Collet-Eckmann maps \citep{BDM10} and hyperbolic maps on $\mathbb{R}^{2}$ \citep{DL08}. For sufficient existence conditions in a general class of dynamical systems, see \citep{CMM00,CMM04}. The convergence of conditionally invariant measures to the invariant measure of the closed system as the hole closes has been investigated by Demers \emph{et al.} \citep{Dem05a,Dem05b, BDM10} using Young towers \citep{You99}. Keller and Liverani \citep{KL09} provide formulae for the variation of the escape rate with the size of small holes.

Another area of recent study is the \emph{connection} between metastability and escape rates. For example \citep{FS10} show that under fairly general conditions, the existence of a ``second eigenvalue'' $\lambda_2\in(0,1)$ allows one to find a set $A$ such that both the escape rate \emph{into} $A$ and \emph{from} A are slower than $-\log\lambda_2$. Specialising to uniformly expanding interval maps, Keller and Liverani~\citep{KL09} consider a $T$ with two invariant disjoint intervals perturbed by a stochastic kernel;  a formula is developed for $\lim_{\epsilon\to 0}(1-\lambda_{2,\epsilon})/\epsilon$ as the amplitude of the kernel shrinks. In \citep{HTW09} a similar setup is considered (with deterministic perturbations) and it is shown that the first and second eigenfunctions, $\varphi$ and $\psi$, of the transfer operator for a metastable system can be written asymptotically as a convex combination of $\varphi_{H}$ and $\varphi_{X\setminus H}$ as the metastability becomes complete stability (no communication). Dolgopyat and Wright~\citep{DW10} show that in the setting of~\citep{HTW09} with $m$ disjoint invariant intervals each supporting a mixing ACIM, one can construct an $m$-state Markov chain with state-to-state jump time distributions that precisely estimate the interval-to-interval jump time distributions of small perturbations of the original non-ergodic map.

So far, the use of Perron-Frobenius operators to study the rate of decay of correlations, metastability, the rate of escape, and the relationships between them has been applied exclusively to the standard test-bed of uniformly expanding interval maps or those that can be modelled by Young towers with exponentially decaying tails\footnote{For example, certain logistic maps~\citep{Dem05a}.}. In this work we move beyond uniform expansivity to study in detail the behaviour of these quantities for a \emph{nonuniformly expanding map} with a finite ACIM, the Pomeau-Manneville (PM) map $T:[0,1]\to[0,1]$ (see equation~(T) in Section~\ref{sec.pmacim}). The natural Young towers have polynomial, rather than exponential, tails and it is well known that the Perron-Frobenius operator lacks a spectral gap on the Banach spaces of functions of bounded variation or H\"older observables typically used in the expanding setting. The reason for the lack of exponential mixing is long escape times from small neighbourhoods
of the origin, due to the presence of an {\em indifferent fixed point\/} at $0$.

We link quantitatively almost-invariance (or metastability) of small neighbourhoods $[0,\epsilon_0]$ with mixing rates of a closed system and the escape rates out of the two open systems $T|_{[0,\epsilon_0]}$ and $T|_{[\epsilon_0,1]}$. Furthermore, we relate the second eigenfunction of the closed system and leading eigenfunctions (ACCIMs) of the open systems in the small hole limit. We connect these ideas with numerical approximations of the ACIMs and ACCIMs via Ulam's method~\citep{Ulam60}. In particular, a small amount of {\em averaging\/} in the neighbourhood of $0$ can induce a {\em spectral gap\/}, stabilizing the eigenvalue~$1$ (of $\p$). Although the form of averaging we introduce looks artificial the resulting analysis explains a phenomenon of {\em spectral gap scaling\/} which is evident when applying Ulam's approximation scheme~\citep{Ulam60} to the calculation of the ACIM for $T$. These results (Theorems \ref{th.4} and \ref{th.5}) may be viewed as intermittent counterparts to spectral stability results for uniformly expanding maps in \citep{BK98,KL99}.

We now give a brief indication of some of our main results.

\paragraph{Two open systems.}
First, we consider two open systems: the dynamics on a small neighbourhood of the origin $[0,\epsilon_0]$ (with trajectories leaving once they expand outside this small interval), and the dynamics on the rest of the domain $[\epsilon_0,1]$ (with trajectories leaving once they enter a small neighbourhood of the origin).
\begin{enumerate}
\item
	\emph{Dynamics on $[0,\epsilon_0]$:} There is no ACCIM for the open dynamics on $[0,\epsilon_0]$. However, a computation of escape rate with respect to a variety of natural probability measures with support on $[0,\epsilon_0]$ (which effectively smooth away the nonexpansion) yields\footnote{We use the following notation: $f(x) = \mathcal{O}(g(x))$ as $x\to a$ if and only if there exist positive real numbers $\delta$ and $M$ such that $|f(x)| \le M|g(x)|$ whenever $|x-a| < \delta$. We write $f(x)\asym g(x)$ if $f(x)=\mathcal{O}(g(x))$ and $g(x)=\mathcal{O}(f(x))$.} $1-\lambda_H \asym \epsilon_0^\alpha$ (see Section~\ref{sec.coarse}).
\item
\emph{Dynamics on $[\epsilon_0,1]$:} The open dynamics on $[\epsilon_0,1]$ does possess an ACCIM (Theorem~\ref{th.1}), which is unique in a certain set of regular measures (Corollary \ref{cor.1}) where the corresponding escape rate is $-\log\lambda_H \approx 1-\lambda_H \asym \epsilon_0$ (Corollary \ref{cor.2}), and is the same as for Lebesgue measure. Furthermore, we show that the ACCIM of this open dynamics converges in $L^1$ to the ACIM of the closed dynamics on $[0,1]$ as $\epsilon_0\to 0$ (Theorem~\ref{th.2}).
\end{enumerate}

\paragraph{The closed system.}
Second, in Section \ref{sec.metastable}, we show that gluing together the two open systems discussed above to form a simple two-state Markov chain predicts a mixing rate or ``second eigenvalue'' whose gap from $1$ scales like the slower of the above two escape rates, namely $\epsilon_0^\alpha$.

We show why this is indeed the case, by explicitly constructing a closed system which includes a smoothed expansion near the origin that mimics the effective smoothing experienced when computing the escape dynamics on $[0,\epsilon_0]$. We prove the existence of a second eigenfunction for the transfer operator of this system and upper and lower bounds for the second eigenvalue (Theorem~\ref{th.4}).

Our method of proof utilises a novel construction combining two Young towers;  one with positive mass and the other with negative mass. This construction elucidates the connection between the second eigenvalue of this closed system and the two escape rates of the two associated open systems, and between the structure of the second eigenfunction of the closed system and the ACCIM on $[\epsilon_0,1]$. Theorem~\ref{th.5} explicitly describes the convergence properties of the ACCIM on $[\epsilon_0,1]$ and the second eigenfunction when they are mapped down from the tower to the unit interval.

\paragraph{Application to numerics.}
A major motivation of the form of smoothing in the neighbourhood of the origin for the PM map is that our theory is applicable to Ulam's method of numerically approximating the Perron-Frobenius operator (sometimes called \emph{coarse-graining}).  Ulam's method is the unique Galerkin-projection method that can also be interpreted as a small random perturbation of the original dynamics (see Theorem 3.7 \citep{koltaithesis}).  In the uniformly hyperbolic setting there is a long history of results showing convergence of Ulam's method for ACIMs \citep{L76,Kel82,F95,DZ96,M98,M01}, and isolated spectral values and their eigenfunctions \citep{BIS95,F97,BK98,KL99}.  Applications of Ulam's method to open systems for the approximation of escape rates and ACCIMs include \citep{Bah06, BB10, BB10a}. Recent work \citep{Mur09} has shown that Ulam's method converges for certain nonuniformly expanding maps.

Ulam's method has been used extensively in a non-rigorous way to detect metastable and almost-invariant sets for nonuniformly hyperbolic, multidimensional systems.  The results obtained are typically very good, when compared with objective measures based on physics \citep{FPET07,FSM10}, even though the outer spectral values used can be considered spurious and due to discretisation effects.

We show that our theoretical results explain the behaviour of Ulam's method when applied to the PM map, including (i) the estimates of escape rate for the open dynamics on $[\epsilon_0,1]$, and (ii) the scaling of the ``spurious'' discretisation-induced second largest eigenvalue of the Ulam matrix for the closed system with the resolution of the projection elements. This present work is a first step in explaining the apparent success of Ulam's method on difficult real-world problems.

\paragraph{Outline of the paper.}
An outline of the paper is as follows. In Section~\ref{sec.pmacim} we describe the family of PM maps and known results concerning the ACIMs, and introduce conditional Perron-Frobenius operators and ACCIMs. In Section~\ref{sec.metastable} we introduce our two-state metastable Markov chain model, discuss the two open systems formed by partitioning [0,1] into $[0,\epsilon_0]$ and $[\epsilon_0,1]$ and state most of our main results. Section~\ref{sec.firsteigenfunctions} contains the Young tower setup and proof of the existence of an ACCIM on the larger domain. Section~\ref{sec.secondeigenfunctions} contains the dual positive-negative Young tower construction and proof of existence of a second eigenfunction for the closed system with bounds on the location of the second eigenvalue. Finally, Section~\ref{sec.numerics} describes Ulam's method or coarse-graining and presents numerical results illustrating the various scalings with hole size.


\section{Invariant measures for Pomeau-Manneville maps and conditionally invariant measures}\label{sec.pmacim}

Let \leb~ denote Lebesgue measure on $[0,1]$. Integrals with respect to \leb~ will sometimes be written $\int f d(\leb) = \int f\,dx$. Let $T:[0,1]\circlearrowleft$ be a Pomeau--Manneville map~\citep{PM80,LSV99} which near $0$ has the form
\begin{equation}\tag{T}
T(x)=x+c_\alpha\,x^{1+\alpha} + g(x)
\end{equation}
where $g$ is $C^2$ and the derivative $g'(x)=o(x^\alpha)$ (in conventional little-o notation\footnote{That is $\lim_{x\to 0}\frac{g'(x)}{x^{\alpha}} = 0$.}). Suppose also that $T$ has two branches, and let $x_0$ be the breakpoint such that $T$ is $1$--$1$ and onto $(0,1)$ on both $(0,x_0)$ and $(x_0,1)$.
We suppose also that $T$ is $C^2$ on both $(x_0,1)$ and $(\epsilon,x_0)$ for every $\epsilon>0$ and that $T'>1$ on both $(0,x_0)$ and $(x_0,1)$.

When $\alpha\in(0,1)$ these maps support a unique absolutely continuous invariant (probability) measure~$\mu$ (ACIM)~(obtainable from a first return map~\citep{Pia80}), the dynamics of $(T,\mu)$ is exact, and $T$ exhibits polynomial decay of correlations~\citep{You99} with rate $\mathcal{O}(k^{1-1/\alpha})$. The density~$\varphi=\frac{d\mu}{dx}$ has a singularity at $0$ (of power law type with exponent $-\alpha$), and arises as a fixed point of the Perron-Frobenius (transfer) operator $\p$ for $T$. The slow decay of correlations occurs because typical orbits of $T$ require anomalously long times to escape from the neighbourhood of $0$.

For small $\epsilon_0$ the set $[0,\epsilon_0)$ is {\em almost-invariant\/}~\citep{DJ99,FD03}. Often, almost-invariant sets are associated with isolated eigenvalues outside the essential spectrum of $\p$ \citep{DFS00}. However, for Pomeau-Manneville type maps\footnote{Indeed any expanding maps with indifferent periodic points.}, the eigenvalue $1$ corresponding to the invariant density is {\em not\/} isolated from the essential spectrum of~$\p$ on any reasonable subspace of $L^1$, leaving no room for isolated ``second'' eigenvalues. We will see in Sections~\ref{sec.secondeigenfunctions} and~\ref{sec.numerics} that the situation is different for certain {\em approximations\/} to $\p$.

\subsection*{Escape rates, conditionally invariant measures and conditional Perron--Frobenius operators}

We use the following standard concepts (see for example~\citep{PY79, DY06}). Let $(X,\nu)$ be a Borel space and let $T:X\rightarrow X$ be a measurable, non-singular transformation\footnote{That is, $\nu\circ T^{-1}$ is absolutely continuous with respect to $\nu$.}. For a measurable set $A$ let $A^{(k)}$ be the $k$-step survivor:
\begin{displaymath}
	A^{(k)} := \{ x\in A\ : \ T^i(x) \in A,\ i=1,\ldots,k \}.
\end{displaymath}
The \emph{escape rate} from $A$ (or into $H:= X\setminus A$), with respect to the measure $\nu$ is given by $-\log\lambda$ where $\log \lambda = \lim_{k\to\infty}\frac{\log\nu(A^{(k)})}{k}$. In absence of ambiguity we may refer to $\lambda$ itself as the geometric escape rate\footnote{$\lambda$ is usually called the \emph{eigenvalue} of $A$ for reasons that will become apparent in the next paragraph.}, or escape rate.

Measures that occur naturally in this setting are those measures, say $\nu_A$ , that satisfy
\begin{displaymath}
	\nu_A\circ (T|_{A})^{-1} = \lambda\nu_A.
\end{displaymath}
Note that $\nu_A$ need be defined only on $A$ but we may extend it to all of $X$. Then $\nu_A(A^{(k)}) = \lambda^k$ and the corresponding escape rate is $-\log \lambda$. These are called \emph{conditionally invariant measures}.

Recall that the \emph{Perron--Frobenius} operator for $T$ acts on integrable functions $\varphi\in L^1(X,\nu)$ according\footnote{Formally speaking, if $X=A\cup H$ (where $H$ is a hole) and $T:A\rightarrow X$, we have $\p_A:L^1(A,\nu|_A)\to L^1(X,\nu)$. The definition of transfer operators between different spaces is is easy: if $T:(Y,\mu)\to(X,\nu)$ is a non-singular transformation ($\mu\circ T^{-1}\ll\nu$), then $\p_T:L^1(Y,\mu)\to L^1(X,\nu)$ is the Frobenius--Perron operator defined by $\p_T\psi = \frac{d}{d\nu}((\psi\,\mu)\circ T^{-1})$} to
\begin{displaymath}
\p\varphi = \frac{d}{d\nu}((\varphi\,\nu)\circ T^{-1})
\end{displaymath}
(where $(\varphi\nu)(E) = \int\varphi\,\chi_E\,d\nu$ and $\frac{d}{d\nu}$ is the usual Radon-Nikodym derivative). Non-negative, normalised fixed points of $\p$ are the densities of {\em absolutely continuous invariant probability measures\/} (ACIMs). The {\em conditional Perron-Frobenius\/} operator (with respect to a measurable set~$A$) is defined by $\p_A\varphi = \chi_A\,\p(\varphi\,\chi_A)$. In much the same way, nonnegative normalised eigenfunctions of $\p_A$ are densities of {\em absolutely continuous conditionally invariant probability measures\/} (ACCIMs). That is if $\p_A\varphi = \lambda\varphi$ then $\nu_A := \varphi\nu$ is conditionally invariant. It is often convenient to define the {\em normalised conditional Perron-Frobenius\/} operator~$\tilde{\p}_A$ which acts according to
\begin{equation}
	\label{eq.normPF}
\tilde{\p}_A\varphi = \frac{\p_A\varphi}{\|\p_A\varphi\|_{L^1}}.
\end{equation}
Then nonnegative fixed points of $\tilde{\p}_A$ are densities of ACCIMs, while the denominator in \eqref{eq.normPF} gives the corresponding escape rate.


\section{Coarse graining, metastability, and a two-state model}
\label{sec.coarse}\label{sec.metastable}

As indicated above, the eigenvalue 1 is not isolated on any reasonable subspace of $L^1([0,1],\nu)$. Nevertheless, approximate, typically compact, versions of $\p$ derived from small random perturbations or numerical schemes possess a spectral gap. One of the main goals in this paper is to explain the variation of this gap with the size of perturbation, or resolution of numerical scheme. To this end, we fix $\epsilon_0\in(0,x_0)$ and partition $[0,1]=I_{\epsilon_0,1}\cup I_{\epsilon_0,2}$ where $I_{\epsilon_0,1}=[0,\epsilon_0]$ and $I_{\epsilon_0,2}=(\epsilon_0,1]$. For all formal results we assume that $\epsilon_0$ is a preimage of $x_0$ (the hole $[0,\epsilon_0]$ is Markov). We begin by summarising known facts and some main results for $T$ and small perturbations of $T$ on these two sets.

\paragraph{Mass of $I_{\epsilon_0,1}$ and $I_{\epsilon_0,2}$.}
The ACIM $\mu$ has $\mu([0,\epsilon_0])\asym \epsilon_0^{1-\alpha}$.

\paragraph{Conditionally invariant measures and rates of escape from $I_{\epsilon_0,j}, j=1,2$.}
We will show that
\begin{itemize}
	\item there exists an ACCIM $\mu_{\epsilon_0}$ on $I_{\epsilon_0,2}$, with eigenvalue $\lambda_{\epsilon_0}$ and density $\varphi_{\epsilon_0}$ bounded away from zero and infinity (shown in Theorem~\ref{th.1} but also a direct consequence of \citep[Theorem 2]{PY79});
	\item $-\log\lambda_{\epsilon_0}$, the rate of escape with respect to the ACCIM\footnote{\label{f.lebrate}And in fact any absolutely continuous measure whose density is bounded away from zero and infinity: Let $\phi\nu$ be such a measure. Then there exist positive constants $b,B$ where $b\varphi_{\epsilon_0} \le \phi \le B\varphi_{\epsilon_0}$ so that $\lim_{k\to\infty}1/k \log \phi\nu(I_{\epsilon_0,2}^{(k)}) \le \lim_{k\to\infty}1/k \log B\mu_{\epsilon_0} (I_{\epsilon_0,2}^{(k)}) = \log\lambda_{\epsilon_0}$. The reverse inequality is similar.} scales as $\epsilon_0$ (Corollary \ref{cor.2});
    \item as $\epsilon_0\to 0$, the ACCIM converges to the ACIM $\mu$ (Theorem~\ref{th.2}) in total variation norm, provided the passage is through a sequence of Markov holes; and
    \item the convergence of the ACCIM on $I_{\epsilon_0,2}$ to the ACIM can be modelled numerically via Ulam's method (Section \ref{sec.numerics}).
\end{itemize}
The situation for $I_{\epsilon_0,1}$ is more problematic as the only recurrent dynamics supported on $I_{\epsilon_0,1}$ is the fixed point at $0$, supporting an invariant $\delta$--measure $\delta_0$. There is no ``natural'' ACCIM\footnote{A natural (cf~\citep{DY06}) ACCIM~$\mu'$ will be one for which $(\tilde{\p}_{I_{\epsilon_0,1}})^k\chi_{I_{\epsilon_0,1}}\to\frac{d\mu'}{dx}$ as $k\to\infty$, where $\tilde{\p}_{I_{\epsilon_0,1}}$ is the normalised conditional Perron-Frobenius operator. Fix $\eta_0\in[0,\epsilon_0)$; define $\eta_k=I_{\epsilon_0,1}\cap T^{-1}\eta_{k-1}$ for each $k>0$ (and similarly for $\epsilon_k$). Then
\begin{displaymath}
\int_0^{\epsilon_0}\chi_{[0,\eta_0)}\,(\tilde{\p}_{I_{\epsilon_0,1}})^k\chi_{I_{\epsilon_0,1}}\,dx
=\frac{\int_0^{\epsilon_0}\chi_{[0,\eta_0)}\circ T^k\,dx}{\int_0^{\epsilon_k}\,dx}=\frac{\eta_k}{\epsilon_k}=1-\frac{\epsilon_k-\eta_k}{\epsilon_k}.
\end{displaymath}
Note that $\epsilon_k\to 0 $ as $k\rightarrow\infty$, $\{\epsilon_k-\epsilon_{k+1}\}$ is a decreasing sequence (since $T$ is expanding) and there is a constant $C$ such that $\epsilon_k-\epsilon_{k+1}\leq C\,\epsilon_k^{1+\alpha}$ for all $k$. Let $L$ be such that $\epsilon_L<\eta_0$. Then $\epsilon_{k+L}<\eta_k$ so
\begin{displaymath}
	\frac{\epsilon_k-\eta_k}{\epsilon_k}<\frac{1}{\epsilon_k}\,\sum_{l=1}^L\epsilon_{k+l-1}-\epsilon_{k+l}\leq \frac{1}{\epsilon_k}\,L\,(\epsilon_k-\epsilon_{k+1})
\leq \frac{1}{\epsilon_k}\,L\,C\,{\epsilon_k}^{1+\alpha}\to 0
\end{displaymath}
as $k\to\infty$. Hence, the measure with density ${\tilde{\p}_{I_{\epsilon_0,1}}}^k\chi_{I_{\epsilon_0,1}}$ converges weakly to $\delta_0$.}.
Moreover, $\leb\circ T^{-k}(I_{\epsilon_0,1})=\mathcal{O}(\min\{\epsilon_0,k^{-1/\alpha}\})$ and\footnote{These facts follow because (in the notation of the previous footnote) there is a constant $c$ such that $\frac{1}{c}\,k^{-1/\alpha}\leq \epsilon_k\leq c\,k^{-1/\alpha}$ and $\frac{d\mu}{dx}\asym x^{-\alpha}$.}
$\mu\circ T^{-k}I_{\epsilon_0,1}=\mathcal{O}(k^{1-1/\alpha})$. Since the escape from $I_{\epsilon_0,1}$ is {\em asymptotically polynomial}, there is no conditionally invariant measure which defines a ``natural'' rate of geometric escape from~$I_{\epsilon_0,1}$.

\paragraph{``Second eigenvalue''.} As pointed out in Section \ref{sec.pmacim}, $\mathcal{P}$ has no well defined ``second eigenvalue''.  We show that
\begin{itemize}
	\item adding a random perturbation on $I_{\epsilon_0,1}$ which averages according to the push-forward of a uniform density on $I_{\epsilon_0,2}\cap T^{-1}I_{\epsilon_0,1}$ induces a second eigenvalue $\lambda_2$ with scaling $1-\lambda_2 \asym \epsilon_0^\alpha$ (Theorems \ref{th.4} and \ref{th.5}); and
	\item coarse-graining the map $T$ by Ulam discretisation with grain size $\epsilon_0$ produces a second eigenvalue $\lambda_2$ with $1-\lambda_2 \asym \epsilon_0^\alpha$ (see Section \ref{sec.numerics}).
\end{itemize}

\subsection*{A two-state metastable model}
Our first approximate version of the Markov operator $\p$ is a crude two-state Markov chain approximation, which nevertheless turns out to be an accurate descriptor of the important dynamics. For any probability measure~$m$ one can construct a $2$-state Markov chain with transition matrix
\begin{displaymath}
	P_{\epsilon_0,m}=\begin{pmatrix}
  1-a_{\epsilon_0} & a_{\epsilon_0} \\
  b_{\epsilon_0} & 1-b_{\epsilon_0} \\
\end{pmatrix},
\end{displaymath}
where $(P_{\epsilon_0,m})_{ij} = \frac{m(I_{\epsilon_0,i}\cap T^{-1}I_{\epsilon_0,j})}{m(I_{\epsilon_0,i})}$. This Markov chain describes the movement between a small neighbourhood of 0 and the rest of the interval. The invariant probability vector is $(\frac{b_{\epsilon_0}}{a_{\epsilon_0}+b_{\epsilon_0}},\frac{a_{\epsilon_0}}{a_{\epsilon_0}+b_{\epsilon_0}})$.

We can view this two-state model in two ways. First, the numbers $1-a_{\epsilon_0}, 1-b_{\epsilon_0}$ are what are sometimes called \emph{almost-invariance ratios} \citep{FD03,F05} and may be interpreted as geometric \emph{escape rates} from the two ``open'' states. Second, as a coarse-grained ``closed'' system, the eigenvalues of $P_{\epsilon_0,m}$ are $1$ and $1-a_{\epsilon_0}-b_{\epsilon_0}$, and so the scaling of the rate of mixing as $\epsilon_0\to 0$ is determined by whichever of $a_{\epsilon_0},b_{\epsilon_0}$ is approaching $0$ most slowly. We now check how well this two-state model reflects the statistics of $T$.

The numbers $a_{\epsilon_0}, b_{\epsilon_0}$ are determined by the choice of measure~$m$ and there are several natural choices:
\begin{enumerate}
\item $m=\leb$ --- this choice requires no special knowledge of the dynamics of $T$ and
corresponds to Ulam's method (discussed below). One has $a_{\epsilon_0} = \frac{\epsilon_0-T^{-1}\epsilon_0}{\epsilon_0}\asym \epsilon_0^\alpha$ and
$\lim_{\epsilon_0\to0}\frac{b_{\epsilon_0}}{\epsilon_0}=\lim_{x\to x_0^+}\frac{1}{T'(x)}$.
\item $m=\mu$ --- although no closed formula is available, the general structure of the invariant density is well understood, and $a_{\epsilon_0}, b_{\epsilon_0}$ scale as when $m=\leb$\footnote{The density $\frac{d\mu}{dx}\asym x^{-\alpha}$
so $a_{\epsilon_0}\asym\frac{\int_{\epsilon_0-c_\alpha\epsilon_0^{1+\alpha}}^{\epsilon_0} x^{-\alpha}\,dx}{\int_0^{\epsilon_0} x^{-\alpha}\,dx}\asym \epsilon_0^{\alpha}$ and $b_{\epsilon_0}\asym \epsilon_0$.}.
\item $m$ is some combination of {\em conditionally invariant measures\/} on each $I_{\epsilon_0,j}$ --- as discussed above, while an ACCIM does exist on  $I_{\epsilon_0,2}$, no ACCIM exists for $I_{\epsilon_0,1}$.
\end{enumerate}
Let us concentrate on options 1 and 2 above, for which we obtain well defined scalings for $a_{\epsilon_0}$ and $b_{\epsilon_0}$.

\paragraph{Mass of $I_{\epsilon_0,1}$ and $I_{\epsilon_0,2}$.}
The mass given to the first state by our two-state model is $\frac{b_{\epsilon_0}}{a_{\epsilon_0}+b_{\epsilon_0}}\asym \epsilon_0^{1-\alpha}$ and matches the ACIM scaling of $\mu([0,\epsilon_0])$.

\paragraph{Rates of escape from $I_{\epsilon_0,j}, j=1,2$.}
The rate of escape from the second state is $-\log(1-b_{\epsilon_0})\approx b_{\epsilon_0} \asym \epsilon_0$ and matches the rate of escape from $I_{\epsilon_0,2}$ with respect to the ACCIM. The rate of escape from the first state is $-\log(1-a_{\epsilon_0})\approx a_{\epsilon_0}\asym \epsilon_0^\alpha$. This is effectively the rate experienced for the map $T$ if the mass distribution on $I_{\epsilon_0,1}$ were reinitialised to $\mu$ or $m$ at each iteration of $T$.

\paragraph{Second eigenvalue of $P_{\epsilon_0,m}$.}
The second eigenvalue is $1-a_{\epsilon_0}-b_{\epsilon_0}$. Since $a_{\epsilon_0}+ b_{\epsilon_0} \asym \epsilon_0^\alpha$, it matches the scaling of the second eigenvalue of the perturbed operator.
\vspace{.5cm}

Thus, this two-state Markov model captures well (i) the relative mass of $I_{\epsilon_0,1}$ and $I_{\epsilon_0,2}$, (ii) provides escape rates from the two states consistent with true or perturbed escape rates for $T$, and (iii) captures well the mixing rate of a perturbed version of $T$. These properties, expressed very clearly with only two states, will carry across to matrices arising from Ulam approximations of $\p$.


\section{Towers, ACIMs and ACCIMs}\label{sec.firsteigenfunctions}

We study the open and closed dynamics of $T$ via a Young tower of returns to an interval away from the indifferent fixed point at $0$. For the ACIM, the construction is standard. For the ACCIM, we puncture the tower for the closed dynamics, and look for a fixed point of the conditional Perron-Frobenius operator on the open tower. Because of the normalisation in the operator, the fixed point must have growing mass concentration as height increases up the tower; some effort is needed to control this growth.

\subsection{Tower set-up and notation}\label{sec.1}

Recall that $T:[0,1]\circlearrowleft$ satisfies (T) and has two, onto $1$--$1$ branches, with a discontinuity at $x_0$. Set $\Delta_0=[x_0,1]$ and let $\Delta$ be constructed as the tower of first returns to $\Delta_0$. That is, let $R(x) = \min\{n>0~:~T^n(x)\in(x_0,1)\}$ and the tower has height $R(x)-1$. Let $\Delta_{0,i}=\{x\in\Delta_0~:~R(x)=i\}$. Then $\Delta=\{(x,\ell)\in\Delta_0\times\Z^+~:~ \ell<R(x)\}$ and the base of the tower $\Delta_0$ is partitioned as $\{\Delta_{0,i}\}_{i=1}^{\infty}$. Because $\{[0,x_0],[x_0,1]\}$ is a Markov partition for $T$, the first return to $\Delta_0$ from the top levels of the tower completely covers $\Delta_0$. This fact is used in our arguments below. The upper levels of the tower are $\Delta_\ell$ partitioned as $\{\Delta_{\ell,i}\}_{i=\ell+1}^{\infty}$. The tower is equipped with a natural measure~$\nu$ (e.g. Lebesgue on $\Delta_0$ lifted by upwards translation). Since $R$ is integrable with respect to $\nu|_{\Delta_0}$, $\nu$ is a finite measure. The tower map is $F: \Delta \circlearrowleft$ with $F(x,\ell)=(x,\ell+1)$ if $\ell<R(x)-1$ and $F(x,R(x))=(T^R(x),0)$ with $F:(\Delta_{0,i}\times\{R(i)\})\rightarrow \Delta_0$ being injective and onto. Assume that $F$ is non-singular and satisfies a variant of the usual regularity condition for $x,y$ on the same level of the tower
\begin{equation}\tag{JF}
\left|\frac{JF(x)}{JF(y)}\right| \leq e^{c\,\beta^{s\circ F(x,y)}}
\end{equation}
where $JF=\frac{d(\nu\circ F)}{d\nu}$ is the Jacobian derivative of $F$, $\beta\in(0,1)$, $c<\infty$ and $s$ is the usual separation time. That is, $s(x,y)$ counts the number of returns to $\Delta_0$ before $x$ and $y$ are in different elements of the partition of $\Delta_0$.

\subsubsection{Truncation and escape from the tower}

Next, for each~$n$ we impose a hole in the tower $H_n:=\cup_{\ell\ge 1, i \ge n+1} \Delta_{\ell,i}$; that is, $H_n$ consists of all elements directly above $H_n^1:=\cup_{i \ge n+1} \Delta_{0,i}$. Note that the highest level of $\Delta\setminus H_n$ is $n-1$ and
\begin{equation}\label{eq.useful1}
H_n^1=(\Delta\setminus H_n) \cap F^{-1}H_n,
\end{equation}
that is $H_n^1$ consists of all the points that fall into the hole in exactly one iteration of the map $F$ (see Figure \ref{fig.twotower}(a)).
If $\p: L^1(\Delta) \circlearrowleft$ is the usual (Perron-Frobenius) transfer operator on the tower
then
\begin{displaymath}
	\p\varphi(x) = \sum_{y\in F^{-1}x}\frac{\varphi(y)}{|JF(y)|}, \quad \text{for } \varphi\in L^1(\Delta)
\end{displaymath}
and for each hole $H_n$ the conditional operator $\p_n : L^1(\Delta) \circlearrowleft$ is given by
\begin{displaymath}
	\p_n\varphi (x) :=\chi_{\Delta\setminus H_n}(x)\,\p(\varphi\cdot \chi_{\Delta\setminus H_n})(x) =
		\chi_{\Delta\setminus H_n}(x)\,\sum_{y\in F^{-1}x\setminus H_n}\frac{\varphi(y)}{|JF(y)|}.
\end{displaymath}

\subsubsection{Distribution of $R$ on the tower}

The distribution of $R$ on $\Delta_0$ is determined by the exponent $\alpha$. Define the sequence $\{x_n\}$ recursively by $x_n=(0,x_0)\cap T^{-1}x_{n-1}$. Standard computations give $x_n \asym n^{-1/\alpha}$ and $\leb(x_{n+1},x_n)\asym n^{-1-1/\alpha}$. Thus $R(x)=i$ precisely when $T(x)\in(x_{i-1},x_{i-2})$ so that $\leb\{x~:~R(x)=i\}=\mathcal{O}(i^{-1-1/\alpha})$ and the tail of the return time distribution is
\begin{displaymath}
\sum_{i\ge n+1} \leb\{R\geq i\}\asym \sum_{i\ge n+1} i^{-1/\alpha} \asym n^{1-1/\alpha}
\end{displaymath}
(giving polynomial decay of correlations with rate $\mathcal{O}(n^{1-1/\alpha})$ by~\citep{You99}) and
\begin{equation}\label{e.sizehole}
	\nu(H_n^1) \asym n^{-1/\alpha}.
\end{equation}

\subsubsection{Satisfaction of (JF) for a tower modelled on $T$}

Let $T$ be a Pomeau--Manneville map satisfying (T) and let $\tau_0(x)=\min\{n>0~:~T^n(x)\in [x_0,1]\}$
denote the first passage/return time to $\Delta_0$. In order to choose $\beta$ such that (JF) is satisfied, note that standard estimates (see for example~\citep{You99}) give a constant $c_0$ such that
\begin{displaymath}
\log\left|\frac{JT^{\tau_0}(x)}{JT^{\tau_0}(y)}\right| \leq c_0\,|T^{\tau_0}(x)-T^{\tau_0}(y)|
\end{displaymath}
when $x,y$ are in the same $1$-$1$ branch of $T^{\tau_0}$. Choosing $\beta<1$ large enough that $|JT(x)|\,\beta >1$ for all $x\in[x_1,1]$ ensures that $|J({T}^{\tau_0})(x)|>\beta^{-1}$ for almost every $x\in (0,1]$. Hence, distances between points are expanded by at least $\beta^{-1}$ on every visit to $\Delta_0$. If $s(x,y)=n$ then $x,y$ lie in the same $1$--$1$ branch of $(T^{\tau_0})^n$ so
\begin{displaymath}
|T^{\tau_0}(x)-T^{\tau_0}(y)|\leq \beta^{n-1}\,|(T^{\tau_0})^n(x)-(T^{\tau_0})^n(y)|\leq \beta^{s\circ F(x,y)}\,\nu(\Delta_0)
\end{displaymath}
and (JF) follows.

\subsection{Existence and uniqueness of ACCIM}\label{sec.E1accim}

In this section we prove the existence and uniqueness of an absolutely continuous conditionally invariant probability measure of $F: \Delta\circlearrowleft$ with hole $H_n$. It should be noted that Theorems 1 and 2 by Pianigiani and Yorke \citep{PY79} may be directly applied in our setting to show these claims. Nonetheless, we duplicate the analogous results on a suitable tower. As well as preparing for the sections that follow, obtaining explicit, uniform bounds on the density of the ACCIMs (independently of hole size) is necessary for our estimates of escape rates, and for showing convergence of conditionally invariant measures to the invariant measure as the hole closes\footnote{The ACCIMs we construct have uniformly bounded densities on $\Delta$, but not when projected back to the interval $[0,1]$.}.

For a fixed constant $C>0$ let $\cone_*$ be a set of \emph{regular} functions in $L^1(\Delta, \nu)$ defined as
\begin{displaymath}
	\cone_* := \left\{ \varphi\in L^1(\Delta,\nu) \ :\ \varphi\ge 0, \ \varphi(x) \le \varphi(y)e^{C\beta^{s(x,y)}} \text{ for a.e. comparable } x,y \right\}.
\end{displaymath}
Above, $x,y\in\Delta$ are considered \emph{comparable} if either they are both in the same cell of the partition $\Delta_{\ell,i}$, or if they are both in $\Delta_0$.

Furthermore for every $n\in\mathbb{Z}^+$ let $\cone_n \subseteq \cone_*$ be a family of regular densities on $\Delta\setminus H_n$, that is
\begin{displaymath}
	\cone_n := \left\{ \varphi\in \cone_*\ : \ \int_{\Delta\setminus H_n} \varphi \ d\nu = 1,\quad \varphi|_{H_n}=0\right\}.
\end{displaymath}
\begin{lemma}\label{lem.1}
	For each $B>0$ the set $\cone_{*}^{B} :=  \{\varphi\in \cone_* \ : \ \|\varphi\|_{L^\infty} \le B\}$ is compact in $L^1(\Delta,\nu)$. In addition for each $n$ there is a $B = B(n)$ such that $\cone_n\subseteq \cone_{*}^{B}$ so that each $\cone_n$ is also compact.
\end{lemma}

\begin{proof}
Define a metric $d_\beta$ on $\Delta$ by $d_\beta(x,y) = \beta^{s(x,y)}$ (for non-comparable $x$ and $y$ set $s(x,y)=0$). For a given $\varphi\in\cone_{*}^{B}$ and any $\epsilon > 0$ choose $\delta = \min(\beta, C^{-1}\log(1+\epsilon/B))$. Take any $x,y\in \Delta$ such that $d_\beta(x,y) < \delta$. As $\delta \le \beta$ we have $s(x,y) \ge 1$ so $x$ and $y$ are in the same cell of the partition of $\Delta$. Then
\begin{displaymath}
		|\varphi(x) - \varphi(y)|  \le \varphi(x)|1-e^{C\beta^{s(x,y)}}|
		\le B(e^{C\beta^{s(x,y)}} - 1)
		\le B(e^{C\delta} - 1) \le \epsilon.
\end{displaymath}
	Thus $\cone_*^{B}$ is equicontinuous. Let $\{\varphi_n\}$ be a sequence in $\cone_*^B$. 
There is no loss of generality in assuming that each $|\varphi_n|\leq B$. Since the metric space $(\Delta,d_\beta)$ is separable, a diagonalisation argument similar to the Arzela-Ascoli theorem establishes the existence of a pointwise convergent subsequence $\{\varphi_{n_j}\}$. Clearly the pointwise limit also belongs to $\cone_*^B$. Since $\nu(\Delta)<\infty$, Lebesgue's dominated convergence theorem implies that $\{\varphi_{n_j}\}$ converges in $L^1(\Delta,\nu)$. This establishes compactness.

	Now consider $\varphi\in\cone_n$. For any  $\Delta_{\ell,i}$, $1\le \ell < i\le n$ by the integral mean value theorem there exists $x^*\in\Delta_{\ell,i}$ such that $\varphi(x^*) = \int_{\Delta_{\ell,i}} \varphi\ d\nu / \nu(\Delta_{\ell,i})$, so
	\begin{displaymath}
		\esssup_{\Delta_{\ell,i}} \varphi \le e^{C\beta}\,\varphi(x^*) = \frac{e^{C\beta}}{\nu(\Delta_{\ell,i})}\int_{\Delta_{\ell,i}}\varphi \ d\nu \le \frac{e^{C\beta}}{\nu(\Delta_{\ell,i})}.
	\end{displaymath}
	Similarly on the base of the tower we obtain $\esssup_{\Delta_0} \varphi \le \frac{e^C}{\nu(\Delta_0)}$. If we choose
	\begin{displaymath}
		B = B(n):= e^C \max\left(\frac{1}{\nu(\Delta_0)}, \max_{1\le\ell<i\le n} \frac{1}{\nu(\Delta_{\ell,i})}\right),
	\end{displaymath}
	then $\cone_n \subseteq \cone_{*}^{B}$ hence $\cone_n$ is also compact.
\end{proof}

\begin{theorem}\label{th.1}
	Let $C\geq c/(1-\beta)$. For each $n\in\mathbb{Z}^{+}$ the normalised conditional operator $\tilde\p_n$ admits a fixed point in $\cone_n$.
\end{theorem}

\begin{proof}
	Note that $C\beta + c \le C$. First, we will show that $\tilde\p_n\cone_n \subseteq \cone_n$ from which a standard fixed point argument will follow. Let $\varphi\in\cone_n$; it suffices to show that $\p_n\varphi\in\cone_*$. Now let $(z,\ell),(w,\ell)\in\Delta_{\ell,i}$ where $1\le \ell < i \le n$ so that both $(z,\ell)$ and $(w,\ell)$ have only one pre-image of $F$, namely $(z,\ell-1)$ and $(w,\ell-1)$. Here, separation time is invariant under $F$ so $s((z,\ell),(w,\ell))=s((z,\ell-1),(w,\ell-1))$. Moreover, $JF=1$ on these levels, since the translation is straight upwards. Hence
	\begin{align*}
		(\p_n\varphi)(z, \ell) &= \frac{\varphi(z,\ell-1)}{JF(z,\ell-1)}\\
		& \le \frac{\varphi(w,\ell-1)e^{C\beta^{s((z,\ell-1),(w,\ell-1))}}}{1}\\
		& = (\p_n\varphi)(w,\ell)e^{C\beta^{s( (z,\ell), (w,\ell))}}.
	\end{align*}
	On the base, the story is different as for each $(z,0), (w,0)\in\Delta_0$ there are $n$ pre-images on the top levels of the tower. For $\ell=0,\dots,n-1$ let $(z_\ell, \ell)\in \Delta_{\ell,\ell+1}$ be such that $F(z_\ell,\ell) = (z,0)$ and similarly for $(w_{\ell}, \ell)$. Now we have $s( (z_{\ell}, \ell), (w_{\ell}, \ell)) = s( (z,0),(w,0)) + 1$ for every $\ell = 0,\cdots,n-1$. Then for any $\varphi\in\cone_n$
\begin{align*}
	(\p_n\varphi)(z,0) &= \sum_{\ell=0}^{n-1}\frac{\varphi(z_{\ell},\ell)}{|JF(z_{\ell},\ell)|}\\
	&\le \sum_{\ell=0}^{n-1}\frac{\varphi(w_{\ell}, \ell)}{|JF(w_{\ell},\ell)|} e^{C\beta^{s( (z_{\ell},\ell), (w_\ell,\ell))}}e^{c\beta^{s\circ F( (z_\ell,l), (w_\ell,\ell))}}\\
	&\le (\p_n\varphi)(w,0) \max_{\ell}  e^{C\beta^{s( (z_{\ell},\ell), (w_\ell,\ell))}}e^{c\beta^{s\circ F( (z_\ell,l), (w_\ell,\ell))}}\\
	&=  (\p_n\varphi)(w,0) \max_{\ell}  e^{C\beta^{s( (z,0), (w,0))+1}}e^{c\beta^{s\circ F( (z_\ell,l), (w_\ell,\ell))}}\\
	&= (\p_n\varphi)(w,0)e^{(C\beta + c)\beta^{s( (z,0), (w,0))}}.
\end{align*}

Since $C\beta + c \le C$ we have $\tilde\p_n\varphi \in \cone_n$ for all $\varphi\in\cone_n$ and therefore $\tilde\p_n\cone_n\subseteq \cone_n$.

It is easy to see that $\cone_n$ is convex as for any $\varphi,\phi\in\cone_n$ and $x,y\in \Delta\setminus H_n$, we have
\begin{align*}
	\lambda\varphi(x) + (1-\lambda)\phi(x) &\le (\lambda \varphi(y) + (1-\lambda)\phi(y))e^{C\beta^{s(x,y)}}.
\end{align*}

Moreover, the operator $\p_n$ is continuous as $\p$ is contractive:
\begin{align*}
	\|\p_n\varphi - \p_n\phi\|_{L^1} &= \| \p ((\varphi-\phi)\cdot\chi_{\Delta\setminus H_n})\|_{L^1}\\
	&\le \|(\varphi-\phi)\cdot\chi_{\Delta\setminus H_n}\|_{L^1} \le   \|\varphi - \phi\|_{L^1}.
\end{align*}

 By the integral mean value theorem and the conditions on $\varphi\in\cone_*^B$,
 \begin{displaymath}
\frac{1}{\nu(H_n^1)}\,\int_{H_n^1}\varphi\,d\nu \leq e^C\frac{1}{\nu(\Delta_0\setminus H_n^1)}\int_{\Delta_0\setminus H_n^1}\varphi\,d\nu.
 \end{displaymath}
Hence
\begin{eqnarray*}
1 = \int_\Delta\,\varphi\,d\nu &=&
\int_{H_n^1}\varphi\,d\nu+\int_{\Delta_0\setminus H_n^1}\varphi\,d\nu + \int_{\Delta\setminus\Delta_0}\varphi\,d\nu\\
&\leq& \left(e^C\,\frac{\nu(H_n^1)}{\nu(\Delta_0\setminus H_n^1)} + \frac{\nu(\Delta_0\setminus H_n^1)}{\nu(\Delta_0\setminus H_n^1)}  \right)\,\int_{\Delta_0\setminus H_n^1}\varphi\,d\nu + \int_{\Delta\setminus\Delta_0}\varphi\,d\nu\\
&<& e^C\,\frac{\nu(\Delta_0)}{\nu(\Delta_0\setminus H_n^1)}\, \left( \int_{\Delta_0\setminus H_n^1}\varphi\,d\nu + \int_{\Delta\setminus \Delta_0} \varphi \ d\nu\right)\\
&=& e^C\,\frac{\nu(\Delta_0)}{\nu(\Delta_0\setminus H_n^1)}\, \int_{\Delta\setminus H_n^1}\mathcal{P}_n\varphi\,d\nu=:\alpha\, \int_{\Delta}\mathcal{P}_n\varphi\,d\nu.
\end{eqnarray*}
Since $1/\|\p_n\varphi\| <\alpha$, the normalisation map $\p_n\varphi \mapsto \p_n\varphi/\|\p_n\varphi\|_{L^1}$ is continuous.
Combining the above results with Lemma~\ref{lem.1} we see that $\cone_n$ is a compact, convex set, invariant under the continuous map $\tilde{\p}_n$. The Schauder Fixed Point Theorem asserts that $\tilde\p_n$ has a fixed point $\varphi_n \in {\cone_n}$.
\end{proof}

We prove uniqueness of $\varphi_n$ in $\cone_n$ below.

\begin{corollary}
	Let $C$ satisfy the hypothesis of Theorem~\ref{th.1}. For each $n\in\mathbb{Z}^+$ there is a unique $\varphi_n \in {\cone_n}$ such that $\p_n\varphi_n = \lambda_n \varphi_n$ where $\lambda_n = \|\p_n\varphi_n\|_{L^1}$. In addition $\varphi_n$ is bounded above and below by positive constants.
	\label{cor.1}
\end{corollary}
\begin{proof}
	Let $\varphi_n\in{\cone_n}$ be a fixed point of $\tilde{\p_n}$. Then $\varphi_n$ satisfies $\p_n \varphi_n = \lambda_n\varphi_n$ where $\lambda_n = \|\p_n\varphi_n\|_{L^1}$. Now if $\essinf\varphi_n = 0$ then the regularity of $\varphi_n$ ensures that $\varphi_n|_{\Delta_{\ell,i}}\equiv 0$ a.e. on some $\Delta_{\ell,i}$. Take any $x\in\Delta_{0,i}$. Then $0 = \lambda_n^{\ell}\varphi_n(F^{\ell}(x)) = (\p_n^{\ell}\varphi_n)(F^{\ell}(x))=\varphi_n(x)$, hence $\varphi_n|_{\Delta_{0,i}}\equiv 0$. All of $\Delta_0$ is comparable so this forces $\varphi_n$ to vanish on $\Delta_0$ and hence on all of $\Delta$. Clearly this is not possible as $\varphi_n$ is a density so necessarily $\essinf\varphi_n > 0$.
	
	Now, since $\varphi_n$ is bounded above\footnote{Note that $\varphi_n\in\cone_n$ and see Lemma~\ref{lem.1}.} and below by positive constants, $-\log\lambda_n$ is the Lebesgue escape rate into the hole $H_n$ so $\lambda_n$ is unique ({\em cf.} footnote~\ref{f.lebrate}).
	
	In the proof of uniqueness of $\varphi_n$ we borrow a technique from \citep{PY79}. Suppose that there is another eigenfunction $\phi_n$ with the same eigenvalue $\lambda_n$. For any $s\in \mathbb{R}$ we are able to construct another eigenfunction $f_s$ of $\p_n$, namely
	\begin{displaymath}
		f_s := s\varphi_n + (1-s)\phi_n.
	\end{displaymath}
Let $\sigma>1$ be the largest real number so that $\essinf f_s \ge 0$ for all $s\in(1,\sigma]$. Then necessarily $\essinf f_{\sigma} = 0$ and $f_{\sigma} = \lim_{s\to\sigma} f_s \in {\cone_n}$. We have already seen in the first part of the proof that this cannot be, hence $\varphi_n$ is unique.
\end{proof}

\begin{corollary}\label{cor.2}
	Let $C$ be such that Theorem~\ref{th.1} holds and let $\varphi_n$ and $\lambda_n$ be as in Corollary \ref{cor.1}. For each $n\in\mathbb{Z}^{+}$ let $\mu_n$ be the measure on $\Delta$ with density $\varphi_n = d\mu_n / d\nu$. Then $\mu_n$ is an absolutely continuous conditionally invariant probability measure for the open system with hole $H_n$. In particular $\lambda_n = 1-\mu_n(H_{n}^{1})$.
\end{corollary}

\begin{proof}
	It is a well known result that nonnegative normalised eigenfunctions of the Perron-Frobenius operator are densities of absolutely continuous conditionally invariant probability measures (see e.g. \citep{PY79, DY06}). As $\mu_n$ is conditionally invariant with eigenvalue $\lambda_n$ we have
	\begin{align*}
		\lambda_n &= \mu_n(F^{-1} (\Delta\setminus H_n) \setminus H_n)\\
		&= \mu_n(F^{-1}(\Delta)\setminus H_n) - \mu_n(F^{-1}(H_n)\setminus H_n) = 1 - \mu_n(H_n^1).
	\end{align*}
\end{proof}

\subsection{Convergence of ACCIM to the ACIM of the closed system}

\begin{lemma}\label{lem.2}
	Let $\varphi_n\in\cone_n$ be as in Theorem~\ref{th.1}. There exist positive constants $a$ and $b$ (independent of $n$) such that $\essinf_{\Delta\setminus H_n} \varphi_n \ge a$ and $\esssup_{\Delta\setminus H_n} \varphi_n \le b$ for all $n\in\mathbb{Z}^+$.
\end{lemma}

\begin{proof}
	Fix $n\in\mathbb{Z}^{+}$ and let $\varphi_{n,i} = \varphi_n|_{\Delta_{0,i}}$ for $1 \le i \le n$ and $\varphi_{n,n+1} = \varphi_n|_{H_n^1}$. First we approximate a lower bound of $\lambda_n$ and then obtain a uniform upper bound for $\lambda_n^{-n}$. We begin with the result of Corollary \ref{cor.2} to obtain
\begin{displaymath}
		\lambda_n = 1 - \int_{H_n^{1}} \varphi_{n} \ d\nu
		 \ge 1 - \nu(H_n^1) e^C \frac{\int_{\Delta_0} \varphi_n \ d\nu}{\nu(\Delta_0)}
		 \ge 1 - \frac{\nu(H_n^1)}{\nu(\Delta_0)}e^C,
\end{displaymath}
	where the first inequality above is a consequence of the integral mean value theorem and the property that $\varphi(x) \le e^C\varphi(y)$ for all $x,y \in \Delta_0$. Now using the fact that $\nu(H_n^1)\cdot n < \nu(\Delta)$ we obtain
\begin{displaymath}
	\lambda_n \ge 1 - \frac{e^C \nu(\Delta)}{n \cdot \nu(\Delta_0)} = 1 - \frac{C'}{n}
\end{displaymath}
for some constant $C' := e^C\nu(\Delta) / \nu(\Delta_0)$ independent of $n$. Next choose $n^{*}\in\mathbb{Z}^+$ so that $C'/n < 1/2$ for all $n\ge n^*$. By the mean value theorem there is a constant $C''$ such that $\log(1-C'/n)\geq -C''/n$ for all $n\geq n^*$ and hence
\begin{equation}
	\lambda_{n}^{-n} \le \left( 1-\frac{C'}{n} \right)^{-n} =e^{-n\,\log(1-C'/n)}\le e^{C''}.
	\label{eq.lambda_n}
\end{equation}
Using the bound on $\lambda_n^{-n}$ and the fact that $\varphi_n$ is an eigenvector of norm $1$ we obtain
\begin{align*}
	1  = \|\varphi_n\|_{L^1} &=  \|\varphi_{n,n+1}\|_{L^1} + \sum_{i=1}^{n} \sum_{j=1}^{i} \lambda_n^{-(j-1)}\|\varphi_{n,i}\|_{L^1}\\
	&\le (\essinf_{\Delta_0}\varphi_n) e^C \left( \nu(H_n^1) + \sum_{i=1}^{n}\sum_{j=1}^{i}\lambda_n^{-(j-1)} \nu(\Delta_{0,i}) \right)\\
	&\le (\essinf_{\Delta_0}\varphi_n)e^C \left( \nu(H_n^1) + \sum_{i=1}^{n} i\cdot e^{C''} \nu(\Delta_{0,i}) \right)\\
	&\le (\essinf_{\Delta_0}\varphi_n)e^{C+C''}\left( \sum_{i=1}^{\infty} i\cdot \nu(\Delta_{0,i}) \right)\\
	& = (\essinf_{\Delta_0}\varphi_n)e^{C+C''}\nu(\Delta).
\end{align*}
Hence we have a uniform lower bound on $\varphi_n$ on $\Delta_0$ and therefore on $\Delta\setminus H_n$ for all $n\ge n^*$:
\begin{displaymath}
	\essinf_{\Delta\setminus H_n} \varphi_n = \essinf_{\Delta_0}\varphi_n \ge \left(e^{C+C''}\nu(\Delta)\right)^{-1} > 0.
\end{displaymath}
With \eqref{eq.lambda_n} we are also able estimate an upper bound of $\varphi_n$, for $n\ge n^*$:
\begin{displaymath}
	\esssup_{\Delta\setminus H_n} \varphi_n \le \lambda_n^{-n} \esssup_{\Delta_0}\varphi_n
	\le e^{C''} \frac{e^C}{\nu(\Delta_0)} \int_{\Delta_0}\varphi_n \ d\nu
	 \le \frac{e^{C+C''}}{\nu(\Delta_0)}.
\end{displaymath}
Since $n^*$ is finite, we conclude that there exist constants $a>0$ and $b>0$ such that $\essinf \varphi_n \ge a$ and $\esssup \varphi_n \le b$ for all $n\ge 1$.
\end{proof}

\begin{corollary}
Let the hypotheses of Theorem~\ref{th.1} hold. There are constants $a$ and $b$ (independent of $n$) such that
\begin{displaymath}
a \leq \lim_{n\rightarrow\infty}\frac{\log\lambda_n}{\nu(H^1_n)}\leq b.
\end{displaymath}
\end{corollary}

\begin{proof}
	Using $\lim_{x\to 1}\frac{\log x}{1-x}=1$ and $1-\lambda_n=\int_{H_n^1}\,\varphi_n\,d\nu$ in conjunction with the result of Lemma~\ref{lem.2} proves the claim.
\end{proof}

\begin{theorem}
	\label{th.2}
	For every positive integer $n$ let $\varphi_n \in \cone_n$ and $\lambda_n < 1$ be as in Corollary \ref{cor.1}. Then $\varphi_n \stackrel{L^1}{\to}\varphi$, where $\varphi$ is the density of the unique absolutely continuous invariant probability measure $\mu$ of the closed system $F:\Delta\circlearrowleft$.
\end{theorem}

\begin{proof}
	The result of Lemma~\ref{lem.2} ensures that all $\varphi_n$ are elements of $\cone_{*}^{b}$, which, as seen in Lemma~\ref{lem.1}, is compact. Hence a subsequence of $\{\varphi_n\}$, say $\{\varphi_{n_i}\}$ converges to some density $\varphi'$. Let $\{\mu_{n_i}\}$ and $\mu'$ be the corresponding measures. Then for any measurable $A \subseteq \Delta$ we have
	\begin{align*}
		\mu'(F^{-1}A) &= \lim_{i\to\infty} \mu_{n_i}(F^{-1}A)\\
    &\leq \lim_{i\to\infty} \mu_{n_i}(F^{-1}(A\setminus H_{n_i}))\\
		&= \lim_{i\to\infty} \lambda_{n_i}\mu_{n_i}(A\setminus H_{n_i})\\
		&= \lim_{i\to\infty} \lambda_{n_i} \mu_{n_i}(A) = \mu'(A).
	\end{align*}
But $\mu'\circ F^{-1}\leq \mu'$ is possible only if $\mu'$ is invariant, therefore $\mu'=\mu$ and $\varphi'=\varphi$ almost everywhere. Hence $\varphi_n \to \varphi$ in $L^1(\Delta,\nu)$ as required.
\end{proof}


\section{On second eigenfunctions}\label{sec.secondeigenfunctions}

The model in the previous section allowed the construction of ACCIMs for certain towers with holes. We now apply these ideas to construct a second eigenfunction for a small perturbation of the Perron-Frobenius operator for the PM map. The new family of towers will be parameterized by $\epsilon$. In connection with the previous section, $\epsilon$ will be approximately the size of the preimage of the hole~$H_n$. Now, orbits dropping into $H_n$ are not lost to the system, but return. The idea is to model the effect of a small random perturbation, analogous to the Ulam discretization; we will use a {\em tower with two bases\/}~$\Delta^\epsilon$. One piece is essentially $\Delta\setminus H_n$, modelling the dynamics of $T$ until leakage into $[0,\mathcal{O}(\epsilon)]$, the other piece will model a smoothed version of the dynamics of $T$ returning from the neighbourhood of $0$. It should be mentioned that in this section, to distinguish between Perron-Frobenius operators of various maps, we may write the corresponding map in the subscript (e.g. $\p_T$). This should not be confused with our previous notation for representing the conditional Perron-Frobenius operator.

\subsection{Construction of a revised tower}

Let $\Delta$ be a tower with a polynomial tail as above:
\begin{equation}\label{e.size}
	\nu(\Delta_{0,i})\asym i^{-1-1/\alpha}.
\end{equation}
Next, fix $\epsilon>0$ and choose $n:=n(\epsilon)$ to be minimal such that $\nu\{\cup_{i\geq n+1}\Delta_{0,i}\}\leq \epsilon$. Then there is a constant $d_1$ such that
\begin{equation}\label{e.tail}
n\leq d_1\epsilon^{-\alpha}.
\end{equation}
Fix $\epsilon_1=\nu(\cup_{i\geq n+1}\Delta_{0,i})$ and $\epsilon_2=\nu(\Delta_{0,n+1})$. Then there are constants $d_2,d_3$ such that
\begin{equation}\label{e.small}
d_2\,\frac{\epsilon_1}{n} \leq \epsilon_2 \leq d_3\,\frac{\epsilon_1}{n}.
\end{equation}
The first tower is $\Depss=\Delta\setminus H_n$ where $H_n=\cup_{i\geq n+1}\cup_{\ell\ge 1}\Delta_{\ell,i}$. Thus, $\Delta_0$ is the base of the tower $\Depss$. Put $\Depss_{0,n+1}=\cup_{i\geq n+1}\Delta_{0,i}$ (the pre-image of the hole $H_n$) and let $\{\Delta_{0,i}\}_{i=1}^n$ (same notation as in Section~\ref{sec.1}) define the remainder of the partition of $\Depss_0$. The base of the second tower~$\Deps$ is identified with $\cup_{i\geq n+1}\Delta_{1,i}$. Let $\iota:(\cup_{i\geq n+1}\Delta_{0,i})\rightarrow [0,\epsilon_1)$ be a Lebesgue measure preserving bijection such that $\iota(\cup_{i> n+1}\Delta_{0,i})=[0,\epsilon_1-\epsilon_2)$ and $\iota(\Delta_{0,n+1})=[\epsilon_1-\epsilon_2,\epsilon_1)$. Put
$\Deps_0=\iota(\cup_{i\geq n+1}\Delta_{0,i})\times\{0\}$. We will replace the action of $F$ on $H_{n}^1=\cup_{i\geq n+1}\Delta_{0,i}$ by $\iota$; instead of direct upwards translation, $H_n^1$ is slid out of the way and regarded as the base of a separate tower. Partition
\begin{displaymath}
\Deps_{0}=\Deps_{0,1}\cup\Deps_{0,n}
\end{displaymath}
where $\Deps_{0,1}=[0,\epsilon_1-\epsilon_2)$ and $\Deps_{0,n}=[\epsilon_1-\epsilon_2,\epsilon_1)$.

The height function is defined as
\begin{displaymath}
h^\epsilon(x) = \left\{\begin{array}{ll} k&x\in \Depss_{0,k}, k\leq n\\
1 & x\in \Depss_{0,{n+1}}\\
1 & x\in \Deps_{0,1}\\
n & x\in\Deps_{0,n}\end{array}\right.
\end{displaymath}
Then $\Delta^\epsilon:=\Deps\cup\Depss$ where $\Depsj=\{(x,\ell)~:~x\in \Depsj_0, \ell<h^\epsilon(x)\}$ for $*\in\{+,-\}$. Points in the tower will be said to be ``in $\Depsj$'' with the obvious meaning. 

The dynamics on $\Depss$ are
\begin{displaymath}
	F^\epsilon(x,\ell) = \left\{\begin{array}{ll} F(x,\ell)&\mbox{if~}(x,\ell)\in\Depss\setminus\Depss_{0,n+1}\\
    (\iota(x),0)& (x,0)\in \Depss_{0,n+1}.\end{array}\right.
\end{displaymath}

Note that when $(x,0)\in\Depss_{0,n+1}$, $F^\epsilon(x,0)\in\Deps_0$.    On $\Deps$ put
\begin{displaymath}
	F^\epsilon(x,\ell) = \left\{\begin{array}{ll} \left(\frac{\epsilon_1}{\epsilon_1-\epsilon_2}\,x,0\right)&(x,\ell)\in\Deps_{0,1}\\
(x,\ell+1)& (x,\ell)\in\Deps_{\ell,n},\ell<n-1\\
F^{n+1}(\iota^{-1}(x),0)& (x,\ell)\in\Deps_{n-1,n}\end{array}\right.
\end{displaymath}
Note that when $(x,n-1)\in\Deps_{n-1,n}$, $F^\epsilon(x,n-1)\in\Depss_0$.

The picture on $\Deps$ is different to the usual tower dynamics and has been designed to allow the analysis of a small random perturbation of the Pomeau-Manneville map $T$. When the return time to $\Delta_0$ is greater or equal to $n$ (corresponding to arrival by orbits of $T$ in a small neighbourhood of $0$), the dynamics of $T$ are replaced by weak affine expansion near $0$.
The natural measure $\nu^\epsilon$ on $\Delta^\epsilon$ is obtained from Lebesgue measure on the interval in the usual way.

The separation time~$s^\epsilon$ on $\Delta^\epsilon$ is defined with respect to the partition of $\Depss_0$ so  the satisfaction of (JF) for $F^\epsilon$ is inherited from the corresponding property of $F$.

\begin{figure}
\begin{center}
 \includegraphics[width=16cm]{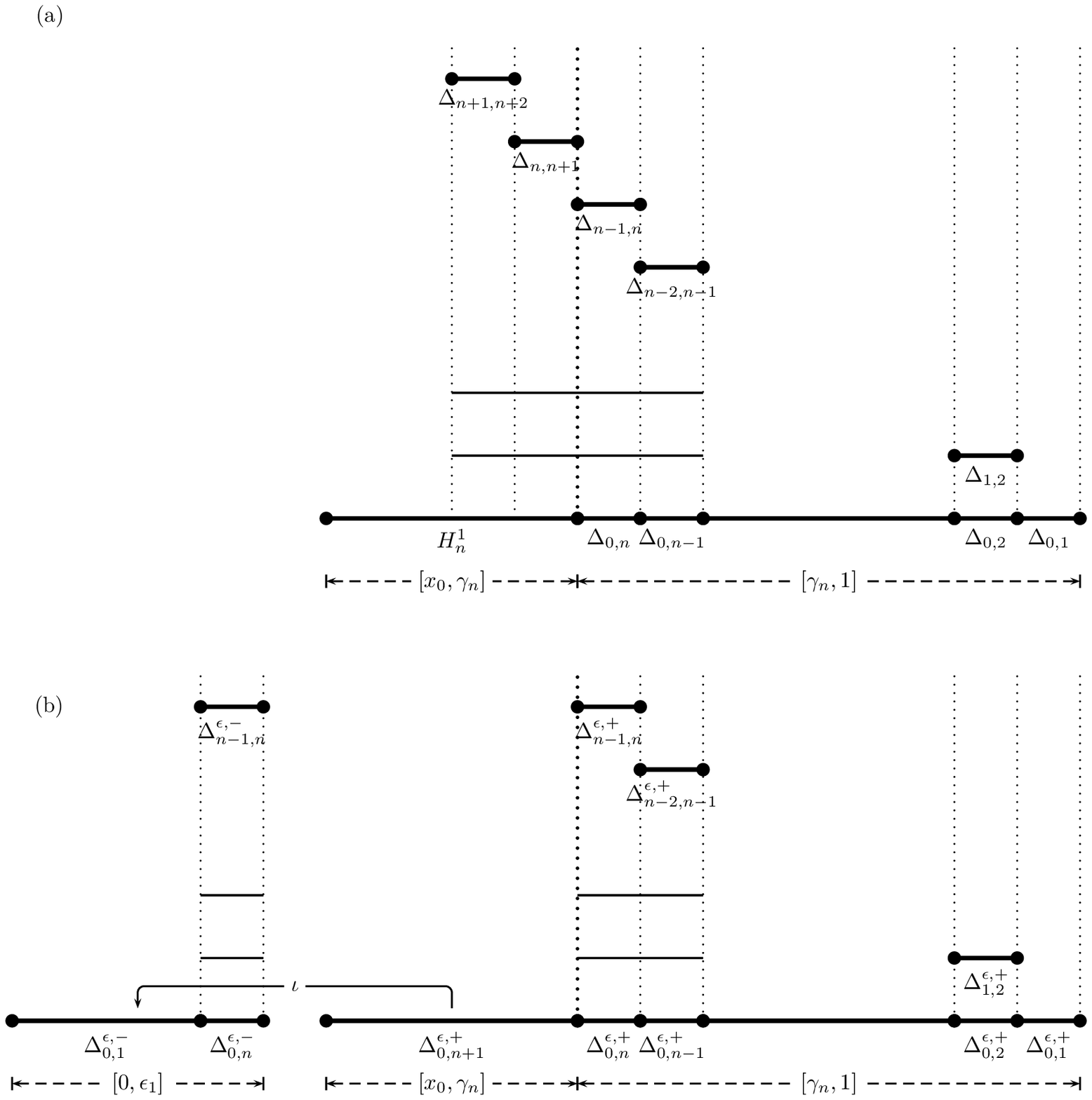}
  \caption{The two towers:  (a) The tower of Section \ref{sec.firsteigenfunctions} with base identified with $[x_0,1]$;  (b) The two towers of Section \ref{sec.secondeigenfunctions} with bases identified with $[0,\epsilon_1]$ and $[x_0,1]$. The left hand tower contains negative mass and the right hand tower contains positive mass.  The map $\iota$ slides the base $\Delta_{0,n+1}^{\epsilon,+}$ across to the base $\Delta_{0}^{\epsilon,-}$.} \label{fig.twotower}
\end{center}
\end{figure}

\subsection{A modified cone and transfer operator}

Fix constants $g,G>0$ and $C^\prime > c/(1-\beta)$. Let $\epsilon>0$ and assume the constants $d_1,d_2,d_3$ and $\alpha$ are all fixed such that $\Delta^\epsilon$ is built with $n$, $\epsilon_1,\epsilon_2$ satisfying (\ref{e.tail}) and (\ref{e.small}).

$\cone^{\epsilon,G,g}$ consists of functions $\psi$ satisfying the conditions (C1.a)-(C2.b):
\begin{enumerate}
\item[(C1.a)] $\psi>0$ a.e. on $\Depss$ and $\int_{\Depss}\psi = \frac{1}{2}$;
\item[(C1.b)] for almost every $z,y\in\Depss_\ell$, $\psi$ satisfies the {\em regularity condition\/},
$$\frac{\psi(z)}{\psi(w)}\le e^{C^\prime\beta^{s^\epsilon(z,w)}};$$
\item[(C1.c)] letting $\psi_\ell=\psi|_{\Depss_\ell}$ ($\ell<n$), the {\em growth condition\/}
$$
\|\psi_\ell\|_\infty \leq G\,\left(1-\frac{g}{n}\right)^{-\ell};$$
holds.
\item[(C2.a)] $\psi<0$ on $\Deps$ and $\int_{\Deps}\psi = \frac{-1}{2}$;
\item[(C2.b)] $\psi$ is piecewise constant on the levels of $\Deps$. Explicitly,
$$\psi|_{\Deps} = a_0\,\chi_{\Deps_0}+\sum_{\ell=1}^{n-1}a_\ell\,\chi_{\Deps_{\ell,n}}$$
and the numbers $\{a_\ell\}_{\ell=0}^{n-1}$ satisfy the {\em growth condition\/}
$$ |a_\ell|\leq |a_0|\Lambda^{-\ell}$$
where
$$\Lambda=\Lambda(\epsilon):=1-2\,\frac{d_3}{n}.$$
\end{enumerate}

Functions in $\cone^{\epsilon,G,g}$ have zero mean, are negative on $\Deps$ and positive on $\Depss$; they have good ``horizontal'' regularity, and a controlled growth rate up the levels of the tower. The more flexible growth rate up levels of the tower in $\Depss$ is made for convenience.

Let $\PFeps$ be the Perron-Frobenius operator for $F^\epsilon$ acting on functions in $\cone^{\epsilon,G,g}$. Unfortunately, $\PFeps\psi$ will typically not be piecewise constant on $\Deps$. To get around this problem, let $\ADeps$ act according to
\begin{equation}\label{e.piepsilon}
\ADeps\psi(x) = \left\{\begin{array}{ll}\frac{\int_{\Deps_0}\,\psi\,d\nu^\epsilon}{\nu^\epsilon(\Deps_0)}&x\in\Deps_0,\\ \psi(x)&\text{otherwise,}\end{array}\right.
\end{equation}
and put $\Leps=\ADeps\,\PFeps$. Notice that $\Leps$ is the transfer operator associated with replacing $F^\epsilon(x)$ with a random perturbation uniformly distributed on $\Deps_0$ whenever $F^\epsilon(x)\in\Deps_0$.
We will show that the normalised version of $\Leps$ has a fixed point in $\cone^{\epsilon,G,g}$; this
is an eigenfunction of $\Leps$ with eigenvalue bounded below by $\Lambda(\epsilon)$. We prepare by establishing some upper and lower bounds on elements of $\cone^{\epsilon,G,g}$.

\begin{lemma}\label{lem.5}
Suppose that $0<\psi:\Depss\rightarrow \mathbb{R}_+$ satisfies $\psi\leq B\,\|\psi\|_{L^1}$ and $\psi(x)/\psi(y)<A$ a.e. if $x,y$ are on the same level of the tower. For all small enough $\epsilon$ there is $\rho>0$ depending on $B$ and $JF$ but not $\epsilon$ or $A$ such that $(\Leps\psi)|_{\Depss_0}\geq \rho\,\|\psi\|_{L^1}/A^2$ a.e.
\end{lemma}

\begin{proof}
There is no loss of generality in assuming $\|\psi\|_{L^1}=1$. Let $R(x)$ be the first return function to $[x_0,1]$ (defined over all of $\Delta_0$), and $\leb$ be Lebesgue measure on the interval. Since $\int R\,d(\leb)<\infty$ there is a finite $K$ such that $\sum_{k\geq K} \leb\{x~:~R(x)> k\} <\frac{1}{2B}$. Let $\Gamma_K=\cup_{K\leq \ell <j}\Depss_{\ell,j}$ and let $\psi_K=\psi|_{\Gamma_K}$. Note that for each $\ell\geq K$, the $\ell$th level of the tower
\begin{displaymath}
\Depss_\ell := \cup_{j=\ell+1}^n\Depss_{\ell,j}\qquad\text{satisfies}\qquad \nu^\epsilon(\Depss_\ell)\leq \leb\{x~:~R(x)>\ell\};
\end{displaymath}
in particular, $\nu^\epsilon(\Gamma_K)\leq \frac{1}{2B}$. Now estimate
\begin{displaymath}
\int \psi_K\,d\nu^\epsilon \leq B \int_{\Gamma_K}\,d\nu^\epsilon \leq \frac{1}{2}.
\end{displaymath}
In particular, $\int_{\Depss\setminus \Gamma_K}\psi \geq \frac{1}{2}$. Let $\epsilon$ be small enough so that $n\geq K$. For each $0\leq k < K$ let $p_k=\nu^\epsilon(\Depss_{k,k+1})/\nu^\epsilon(\Depss_k)$ and let $p=\min_{0\leq k<K}p_k$. Then $\nu^\epsilon(\Depss_{k,k+1}) \geq p_k\,\nu^\epsilon(\Depss_k)\geq p\,\nu^\epsilon(\Depss_k)$ so that (using the integral mean value theorem)
\begin{displaymath}
\int_{\Depss_{k,k+1}}\psi\,d\nu^\epsilon \geq \frac{p}{A}\,\int_{\Depss_k}\psi\,d\nu^\epsilon.
\end{displaymath}
Now,
\begin{displaymath}
\int_{\Depss_0} \Leps\psi\,d\nu^\epsilon
    = \sum_{k=0}^{n-1} \int_{\Depss_{k,k+1}}\,\psi\,d\nu^\epsilon
    \geq \sum_{k=0}^{K-1} \int_{\Depss_{k,k+1}}\,\psi\,d\nu^\epsilon
    \geq \frac{p}{A}\sum_{k=0}^{K-1}\int_{\Depss_k}\psi\,d\nu^\epsilon
    \geq \frac{p}{A}\,\frac{1}{2}.
\end{displaymath}
By (JF) and the regularity of $\psi$, for $x,y\in\Depss_0$ we have
\begin{displaymath}
 \frac{\Leps\psi(x)}{\Leps\psi(y)} =\frac{\ADeps\PFeps\psi(x)}{\ADeps\PFeps\psi(y)}= \frac{\PFeps\psi(x)}{\PFeps\psi(y)} \leq A\,e^c
\end{displaymath}
so that
\begin{displaymath}
(\Leps\psi)|_{\Depss_0}\geq \frac{e^{-c}}{A}\frac{\int_{\Depss_0}\Leps\psi\,d\nu^\epsilon}{\nu^\epsilon(\Depss_0)}
\geq \frac{p\,e^{-c}}{2\,A^2\,\nu(\Delta_0)}.
\end{displaymath}
\end{proof}

\begin{lemma}\label{lem.3}
Let $\epsilon$ and $n$ be as in the definition of $\Delta^\epsilon$ and suppose that $\psi\in\cone^{\epsilon,G,g}$ (with $\{a_\ell\}_{\ell=0}^{n-1}$ as in (C2.b)). There are constants $0<\delta,D<\infty$ (independent of $\epsilon$) such that $|a_\ell|\leq D\,|a_0|$ for each $\ell<n$, and $|a_0|\,\epsilon_1 \in [\delta,1/2]$.
\end{lemma}

\begin{proof}
Since $\Lambda=1-\frac{2\,d_3}{n}$, similar to the proof of Lemma~\ref{lem.2}, $\Lambda^{-n}\leq D$ for some $D$ independent of $n$. Hence, $|a_\ell|\leq \Lambda^{-\ell}|a_0| \leq D\,|a_0|$. Next, the conditions on $\psi$, and equation~(\ref{e.small}), show that
\begin{displaymath}
\frac{1}{2} = \int_{\Deps}|\psi| = |a_0|\,\epsilon_1 + \sum_{k=1}^{n-1} |a_k|\,\epsilon_2
\leq |a_0|\,\epsilon_1\,\left(1 + \sum_{k=1}^{n-1}D\frac{d_3}{n}\right)< |a_0|\,\epsilon_1\,(1+D\,d_3).
\end{displaymath}
On the other hand,
\begin{displaymath}
\frac{1}{2}=\int_{\Deps}|\psi|\geq \int_{\Deps_0}|\psi| = |a_0|\,\epsilon_1.
\end{displaymath}
\end{proof}

\begin{lemma}\label{lem.4}
Let $\epsilon>0$ be fixed. Each set $\cone^{\epsilon,G,g}$ is relatively compact in $L^1(\Delta^\epsilon,\nu^\epsilon)$.
\end{lemma}

\begin{proof}
Each function $\psi\in\cone^\epsilon$ can be decomposed as $\psi|_{\Deps}+\psi|_{\Depss}$. In view of (C1.a-b), for each $\psi$, $2\,\psi|_{\Depss}\in\cone_n$. Hence, by Lemma~\ref{lem.1}, if $\{\psi_k\}$ is a sequence in $\cone^{\epsilon,G,g}$ then $\{\psi_k|_{\Depss}\}$ has a Cauchy subsequence. By Lemma~\ref{lem.3}, and (C2.b), if $\psi\in\cone^{\epsilon,G,g}$ and $z\in\Deps$ then
\begin{displaymath}
|\psi(z)| \leq \max_{\ell<n}|a_\ell| \leq D\,|a_0|\leq\frac{D}{2\epsilon_1}.
\end{displaymath}
Using this bound, the fact that each $\psi_k|_{\Deps}$ is piecewise constant (condition (C2.b)) and the Heine-Borel theorem, there is a further subsequence $\{\psi_{k_l}\}$ such that $\{\psi_{k_l}|_{\Deps}\}$ is also Cauchy.
\end{proof}

\subsection{Existence of second eigenfunctions and eigenvalue scaling}

Define the {\em normalised operator\/} by
\begin{displaymath}
\tLeps\psi := \frac{\Leps\psi}{\|\Leps\psi\|_{L^1}}.
\end{displaymath}

\begin{theorem}\label{th.4}
There are constants $g$ and $G$ (independent of $\epsilon$) such that, $\tLeps\cone^{\epsilon,G,g}\subseteq\cone^{\epsilon,G,g}$ for all sufficiently small $\epsilon>0$. In particular, the operator $\Leps$ has an eigenvector $\psi^\epsilon\in\cone^{\epsilon,G,g}$ with associated eigenvalue $\lambda^\epsilon$ such that $1-\lambda^\epsilon\in(\frac{\epsilon_2}{\epsilon_1},\frac{2\,\epsilon_2}{\epsilon_1})$.
\end{theorem}

{\em Comments on the proof:\/}
Most of the work is in showing that $\tLeps$ preserves $\cone^{\epsilon,G,g}$. Mass of $\psi$ is exchanged between $\Depss$ and $\Deps$ in two ways by $\Leps$: some fraction of ``positively signed'' mass of $\psi$ is transported from $\Depss$ to $\Deps$ --- this corresponds to points falling into the hole $H_n$. The total transferred in this way is $\int_{\Depss_{0,n+1}}\,\psi$, and it arrives under one
iteration of $F$ on $\Depss_{0,n+1}$ to $\Deps_0$ (the base of the other part of the tower). Correspondingly, ``negatively signed'' mass from level $\Deps_{n-1,n}$ is mapped onto $\Depss_0$. The proof involves bounding the amount of mass transferred in these ways, and controlling the distorting effects on regularity.

\begin{proof}
Let $\psi\in\cone^{\epsilon,G,g}$. The choice of $G$ is made in (G) below, and $g$ is fixed near the end of the proof. The conditions (C1.b) and (C2.b) are projective, so will hold for $\tLeps\psi$ if they hold for $\Leps\psi$. Since $\Feps$ acts by upwards translation on $\Deps$ and $\ADeps$ acts by averaging on $\Deps_0$, it is clear that $(\Leps\psi)|_{\Deps}$ is piecewise constant.
\newline
{\em Preservation of the growth condition on $\Deps$:\/}
We first consider $(\Leps\psi)|_{\Deps}$ (note that $\psi|_{\Deps}<0$). Let $\{a_\ell\}$ be as in (C2.b) and put $a_{\ell}^\prime=(\Leps\psi)|_{\Deps_\ell}$. Since the action of $F^\epsilon$ on $\Deps_{\ell,n}$ is direct upwards translation, and the averaging $\ADeps$ occurs only over $\Deps_0$, for $\ell>0$
\begin{equation}\label{eq.pf1}
|a_\ell^\prime| = \left|(\Leps\psi)|_{\Deps_\ell}\right| = \left|\psi|_{\Deps_{\ell-1}}\right| =|a_{\ell-1}|\leq \Lambda^{-\ell}(\Lambda\,|a_0|).
\end{equation}
We now show that $\Lambda\,|a_0| \leq |a_0^\prime|$. Since $(F^\epsilon)^{-1}\Deps_0 = \Deps_{0,1}\cup\Depss_{0,n+1}$, one has
\begin{equation}\label{eq.pf3}
\int_{\Deps_0}\Leps\psi\,d\nu^\epsilon = \int_{\Deps_0}\PFeps\,\psi\,d\nu^\epsilon
=\underbrace{\int_{\Deps_{0,1}}\,\psi\,d\nu^\epsilon}_{(I)} + \underbrace{\int_{\Depss_{0,n+1}}\,\psi\,d\nu^\epsilon}_{(II)}.
\end{equation}
Now
\begin{displaymath}
(I)=\int_0^{x_n}a_0\,dx=x_n\,a_0=(\epsilon_1-\epsilon_2)\,a_0.
\end{displaymath}
As for $(II)$, for $x,y\in\Depss_0$,
\begin{equation}\tag{G}
\frac{\psi(x)}{\psi(y)}\leq e^{C^\prime}\qquad\text{so}\qquad
\psi(x) \leq e^{C^\prime} \frac{\int_{\Depss_0}\psi\,d\nu^\epsilon}{\nu^\epsilon(\Depss_0)} \leq \frac{e^{C^\prime}\,\int_{\Delta^{\epsilon,+}}\psi\,d\nu^\epsilon}{\nu^\epsilon(\Delta_0^{\epsilon,+})} = \frac{e^{C^\prime}\,1/2}{\nu(\Delta_0)}=:G
\end{equation}
where $G$ is independent of $x$ (and $\epsilon$). Since $\nu^\epsilon(\Depss_{0,n+1})=\epsilon_1$, $(II) \leq G\,\epsilon_1$. To complete the estimate on $(II)$, use $\delta$ from Lemma~\ref{lem.3} and equations~(\ref{e.small}) and (\ref{e.tail}):
\begin{displaymath}
(II)\leq G\epsilon_1 = \frac{G}{\delta}\,\frac{{\epsilon_1}^2}{\epsilon_2}\,\left(\frac{\delta}{\epsilon_1}\,\epsilon_2\right)
\leq\frac{G}{\delta} \,\frac{n\,\epsilon_1}{d_2}\,\left(|a_0|\,\epsilon_2\right)
\leq\frac{G}{\delta} \,\frac{d_1\,{\epsilon_1}^{1-\alpha}}{d_2}\,\left(|a_0|\,\epsilon_2\right) \leq |a_0|\,\epsilon_2
\end{displaymath}
for sufficiently small $\epsilon$. Hence, by~(\ref{eq.pf3}),
\begin{displaymath}
\int_{\Deps_0}\Leps\psi\,d\nu^\epsilon = (I) + (II) \leq a_0\,(\epsilon_1-\epsilon_2) + |a_0|\,\epsilon_2=a_0(\epsilon_1-2\,\epsilon_2)\leq-|a_0|\,\epsilon_1\,\Lambda(\epsilon).
\end{displaymath}
Now,
\begin{equation}\label{eq.pf4}
-a_0^\prime=-(\Leps\psi)|_{\Deps_0} = \frac{-\int_{\Deps_0}\Leps\psi\,d\nu^\epsilon}{\epsilon_1} \geq |a_0|\,\Lambda(\epsilon).
\end{equation}
In view of (\ref{eq.pf1}) and $(\ref{eq.pf4})$, $(\Leps\psi)|_\Deps<0$ and the growth condition (C2.b) is satisfied by $\Leps\psi$.
\newline
\smallskip
{\em Regularity control on $\Depss$:\/} Note that $\ADeps|_{\Depss}$ is the identity, so $(\Leps\psi)|_{\Depss}=(\PFeps\psi)|_{\Depss}$.
Now, if $z,w\in\Depss_\ell$ where $\ell>0$ then
\begin{displaymath}
\frac{\Leps\psi(z,\ell)}{\Leps\psi(w,\ell)} = \frac{\PFeps\psi(z,\ell)}{\PFeps\psi(w,\ell)} = \frac{\psi(z,\ell-1)}{\psi(w,\ell-1)}\leq e^{C^\prime\beta^{s^\epsilon((z,\ell),(w,\ell))}}
\end{displaymath}
since $s^\epsilon((z,\ell),(w,\ell))=s^\epsilon((z,\ell-1),(w,\ell-1))$. The regularity estimate on $\Depss_0$ is more complicated. Write $\psi=\phi_+ + \phi_-$ where $\phi_*=\psi|_{\Depsj}$ for $*\in\{+,-\}$. Let $z,w\in\Depss_0$ and let $z_k\in \Depss_{k-1,k}$ ($k=1,\ldots,n$) be such that $F^\epsilon(z_k)=z$ and let $z_0\in\Deps_{n-1,n}$ be such that $F^\epsilon(z_0)=z$; use similar notation for $\{w_k\}_{k=0}^n$. Now,
\begin{eqnarray*}
\PFeps\psi(z) &=& \sum_{k=1}^n\phi_+(z_k)/|JF^\epsilon(z_k)| + \phi_-(z_0)/|JF^\epsilon(z_0)|\\
&\leq& \sum_{k=1}^n\phi_+(w_k)e^{C^\prime\,\beta^{s^\epsilon(z_k,w_k)}}\,e^{c\,\beta^{s^\epsilon(F^\epsilon(z_k),F^\epsilon(w_k))}}/|JF^\epsilon(w_k)| + \phi_-(z_0)/|JF^\epsilon(z_0)|\\
&\leq& \sum_{k=1}^n\phi_+(w_k)e^{C^\prime\,\beta^{s^\epsilon(z_k,w_k)}}\,e^{c\,\beta^{s^\epsilon(F^\epsilon(z_k),F^\epsilon(w_k))}}/|JF^\epsilon(w_k)| \\
&&\qquad\qquad+ \phi_-(w_0)e^{-c\,\beta^{s^\epsilon(F^\epsilon(z_0),F^\epsilon(w_0))}}/|JF^\epsilon(w_0)|\\
&=&e^{(C^\prime\,\beta+c)\,\beta^{s^\epsilon(z,w)}}\,\left(\sum_{k=1}^n\phi_+(w_k)/|JF^\epsilon(w_k)| + \phi_-(w_0)/|JF^\epsilon(w_0)|\right)
\\&&\qquad\qquad- e^{(C^\prime\,\beta+c)\,\beta^{s^\epsilon(z,w)}}\,(\phi_-(w_0)/|JF^\epsilon(w_0)|)\,\left(1-e^{-(C^\prime\,\beta+2\,c)\,\beta^{s^\epsilon(z,w)}}\right)\\
&\leq&e^{(C^\prime\,\beta+c)\,\beta^{s^\epsilon(z,w)}}\,\left(\p\psi(w) + (|\phi_-(w_0)|/|JF^\epsilon(w_0)|)\, (C^\prime\,\beta+2\,c)\,\beta^{s^\epsilon(z,w)}\right)
\end{eqnarray*}
using the regularity conditions on $\phi_+=\psi|_{\Depss}$ and (JF), the fact that $\phi_-(z_0)=\phi_-(w_0)=a_{n-1}<0$, and $s^\epsilon(z_k,w_k)=s^\epsilon(F^\epsilon(z_k),F^\epsilon(w_k))+1=s^\epsilon(z,w)+1$. Thus,
\begin{equation}\label{eq.pf2}
\frac{\PFeps\psi(z)}{\PFeps\psi(w)}
\leq e^{(C^\prime\,\beta+c)\,\beta^{s^\epsilon(z,w)}}\,\left( 1+ \underbrace{\frac{|\phi_-(w_0)|/|JF^\epsilon(w_0)|}{\PFeps\psi(w)}\,(C^\prime\,\beta+2\,c)}_{(III)}\,\beta^{s^\epsilon(z,w)}\right).
\end{equation}
Since $C^\prime>\frac{c}{1-\beta}$, the constant $b:=C^\prime-(C^\prime\,\beta+c)>0$.
\newline
{\em Claim:\/} When $\epsilon$ is small enough, $(III)<b$.
\newline
{\em Proof of claim:\/} By Lemma~\ref{lem.3},
\begin{displaymath}
|\phi_-(y_0)|=|a_{n-1}|\leq D\,|a_0| \leq \frac{D}{2\,\epsilon_1}
\end{displaymath}
while
\begin{displaymath}
|JF^\epsilon(w_0)| \geq e^{-c}\frac{\nu^\epsilon(F^\epsilon(\Deps_{n-1,n}))}{\nu^\epsilon(\Deps_{n-1,n})}=e^{-c}\,\frac{\leb[x_0,1]}{\epsilon_2}\geq d_4\,\frac{n}{\epsilon_1}
\end{displaymath}
for a suitable constant $d_4$. Hence $\frac{|\phi_-(w_0)|}{|JF^\epsilon(w_0)|}\leq\frac{D}{2\,d_4}\,\frac{1}{n}$. On the other hand, $\phi_+=\psi|_{\Depss}$ is uniformly bounded\footnote{By (C1.c), $\psi|_{\Depss}\leq B_\epsilon:= G(1-g/n)^{-n}$. $B_\epsilon$ is bounded uniformly in $\epsilon$ by an argument similar to the proof
of Lemma~\ref{lem.2}.},
so by Lemma~\ref{lem.5} ($A=e^{C^\prime}$) there is a constant $d_5$ such that $(\PFeps\phi_+)|_{\Depss} \geq d_5$. Hence $\PFeps\psi(w) = \PFeps\phi_+(w) + \PFeps\phi_-(w) = \PFeps\phi_+(w)-|\phi_-(w_0)|/|JF^\epsilon(w_0)|\geq d_5- \frac{D}{2\,d_4\,n}$ and
\begin{displaymath}
(III) \leq \frac{\frac{D}{2\,d_4}}{n- \frac{D}{2\,d_4}}\,(C^\prime\,\beta+2\,c) <b
\end{displaymath}
for small enough $\epsilon$.\qquad$\Box$\newline
Now using the claim and~(\ref{eq.pf2}),
\begin{displaymath}
\frac{\PFeps\psi(z)}{\PFeps\psi(w)} \leq e^{(C^\prime\,\beta+c)\,\beta^{s^\epsilon(z,w)}}\,\left( 1+ b\,\beta^{s^\epsilon(z,w)}\right)
\leq e^{(C^\prime\,\beta+c)\,\beta^{s^\epsilon(z,w)}}\,e^{b\,\beta^{s^\epsilon(z,w)}}=e^{C^\prime\,\beta^{s^\epsilon(z,w)}};
\end{displaymath}
that is, (C1.b) is satisfied.\newline
\smallskip
{\em Control of $\|\Leps\psi\|_{L^1}$:\/} From the regularity estimates above, when $\epsilon$ is  sufficiently small $(\Leps\psi)|_{\Deps}<0$ and $(\Leps\psi)|_{\Depss}>0$. Since $\Leps$ is a Markov operator, it preserves integrals, so $\int\Leps\psi\,d\nu^\epsilon=0$ and
\begin{displaymath}
\|(\Leps\psi)|_{\Deps}\|_{L^1}=-\int_{\Deps}\Leps\psi\,d\nu^\epsilon=\int_{\Depss}\Leps\psi\,d\nu^\epsilon =\|(\Leps\psi)|_{\Depss}\|_{L^1}
\end{displaymath}
and all terms in the above expression are equal to $\frac{1}{2}\|\Leps\psi\|$. Now,
\begin{eqnarray*}
\|\psi|_{\Deps}\|_{L^1}-\|(\Leps\psi)|_{\Deps}\|_{L^1}&=&\text{total mass exchanged between $\Deps$ and $\Depss$}\\
&=&\int_{\Deps_{n-1,n}}\,|\psi| + \int_{\Depss_{0,n+1}}|\psi|\\
&\leq& |a_{n-1}|\,\epsilon_2 + \epsilon_1\,G\qquad\text{(by the argument bounding (II) above)}\\
&\leq& |a_0|\,D\,\epsilon_2 + \epsilon_1\,G\qquad\text{(by Lemma~\ref{lem.3})}\\
&\leq& \frac{D}{2\epsilon_1}\epsilon_2 + \epsilon_1\,G\qquad\text{(by Lemma~\ref{lem.3})}\\
&\leq& \frac{1}{2\,n}\,(D\,d_3+2\,d_1\,(\epsilon_1)^{1-\alpha}G)
\end{eqnarray*}
(in view of (\ref{e.tail}) and (\ref{e.small})). Putting $g:=D\,d_3+2\,d_1G$ we have
\begin{displaymath}
\|\Leps\psi\|_{L^1} = 2\,\|(\Leps\psi)|_{\Deps}\|_{L^1}
=1-2\,\left(\|\psi|_{\Deps}\|_{L^1}-\|(\Leps\psi)|_{\Deps}\|_{L^1}\right)\geq 1 - \frac{g}{n}.
\end{displaymath}
{\em Control of growth up $\Depss$:\/} The fact that $(\tLeps\psi)|_{\Depss_{0}}\leq G$ follows from (G), since $\|\tLeps\psi\|_{L^1}=1$. Higher up the tower,
\begin{displaymath}
(\tLeps\psi)|_{\Depss_\ell} = \frac{1}{\|\Leps\psi\|_{L^1}}\,(\PFeps\psi)|_{\Depss_\ell}
\leq (1-g/n)^{-1}(\PFeps\psi)|_{\Depss_\ell}\leq (1-g/n)^{-1}\,\max_{\Depss_{\ell-1}}\psi,
\end{displaymath}
since $\PFeps$ acts on $\Depss$ by upwards translation. Since $\max_{\Depss_{\ell-1}}\psi\leq (1-g/n)^{-(\ell-1)}\,G$, the growth condition (C1.c) is satisfied. This completes the proof that $\tLeps\cone^{\epsilon,G,g}\subset \cone^{\epsilon,G,g}$.\smallskip

Finally, the Schauder Fixed Point Theorem in conjunction with Lemma~\ref{lem.4} gives the existence of a fixed point $\psi^\epsilon$ of $\tLeps$ in $\cone^{\epsilon,G,g}$. Evidently, $\Leps\psi^\epsilon = \|\Leps\psi^\epsilon\|_{L^1}\,\psi^\epsilon$, so $\psi^\epsilon$ is an eigenvector for $\Leps$ and $\lambda^\epsilon=\|\Leps\psi^\epsilon\|_{L^1}$. The upper bound on $\lambda^\epsilon$ is obtained via the estimate $(I)$ and~(\ref{eq.pf3}):
\begin{displaymath}
\lambda^\epsilon\,a_0 = (\Leps\psi^\epsilon)|_{\Deps_0} = \frac{\int_{\Deps_0}\psi^\epsilon\,d\nu^\epsilon}{\nu^\epsilon(\Depss_0)}
=\frac{1}{\epsilon_1}\,((I) + (II)) \geq \frac{(I)}{\epsilon_1}=(1-\epsilon_2/\epsilon_1)\,a_0
\end{displaymath}
(since $(II)\geq 0$); divide by $a_0<0$. An immediate lower bound on $\lambda_\epsilon$ is given by $(1-g/n)$, however a more careful estimate using
\begin{displaymath}
(\Lambda(\epsilon))^{-1}\,|a_0|\geq |a_1| = |\psi^\epsilon|_{\Deps_{1,n}}| = (\lambda^\epsilon)^{-1}|(\Leps\psi^\epsilon)|_{\Deps_{1,n}}|
= (\lambda^\epsilon)^{-1}|\psi^\epsilon|_{\Deps_{0,n}}| = (\lambda^\epsilon)^{-1}|a_0|
\end{displaymath}
gives the lower bound $\lambda^\epsilon\geq \Lambda(\epsilon)$.
\end{proof}

Simplifications to the proof of Theorem~\ref{th.4} can be made to establish the existence of a fixed point $\varphi^\epsilon$ of $\Leps$ with the properties
\begin{itemize}
\item $\varphi^\epsilon>0$ a.e. and $\int_{\Delta^\epsilon} \varphi^\epsilon\,d\nu^\epsilon =1$
\item $\varphi^\epsilon|_{\Deps}=\sum_{\ell=0}^{n-1}c_\ell\chi_{\Deps_\ell}$ ($c_\ell$ are constants depending on $\epsilon$)
\item $\frac{\varphi^\epsilon(z)}{\varphi^\epsilon(w)}\leq e^{C'\beta^{s^\epsilon(z,w)}}$ for almost every $z,w\in\Depss_\ell$ (where $C'\geq c/(1-\beta)$)
\end{itemize}
Since $\Leps\varphi^\epsilon=\varphi^\epsilon$, all the $c_\ell$ are equal to a single constant $c^\epsilon$ and\footnote{By balancing the mass exchange between $\Deps$ and $\Depss$.} $c^\epsilon\asym \frac{\epsilon_1}{\epsilon_2}$. Hence $\int_{\Deps}\varphi^\epsilon\,d\nu^\epsilon = c^\epsilon\,(\epsilon_1+(n-1)\,\epsilon_2)\asym\epsilon^{1-\alpha}$. If $\varphi^\epsilon|_{\Depss}$ is embedded as a function on $\Delta^\epsilon$ then Lemma~\ref{lem.1} can be applied (as in Theorem~\ref{th.2}) to establish convergence to the density of the unique ACIM on $\Delta$. Explicitly, let $\varphi\in L^1(\Delta,\nu)$
be the density of the (unique) ACIM for $F$. Then
\begin{equation}\label{e.towercgc}
\varphi^\epsilon|_{\Depss}\stackrel{L^1}{\rightarrow}\varphi\qquad\mbox{and both}\qquad
\|\varphi^\epsilon|_{\Deps}\|_{L^1(\Delta^\epsilon)}, \|\varphi|_{\Delta\setminus\Depss}\|_{L^1(\Delta)}\to 0
\end{equation}
as $\epsilon\to 0$.

\subsection{Realisation of the second eigenfunction for the PM map}

When $\Delta$ arises as the first return tower to $[x_0,1]$ for a PM map $T$, the preimage of the hole $H_n$ is $H_n^1\equiv [x_0,\gamma_n)$ (for some $\gamma_n$ such that $T(\gamma_n)=x_{n-1}$---the $(n-1)$st preimage of $x_0$). There is also $\gamma_{n+1}$ such that $T(\gamma_{n+1})=x_n$ and $\gamma_n-\gamma_{n+1}=\epsilon_2$ while $\gamma_n-x_0=\epsilon_1$. The map $\iota$ translates $[x_0,\gamma_n)$ back to $[0,\epsilon_1)$ with $\iota(x)=x-x_0$ (instead of applying $T$) and the dynamics of $F^\epsilon$ essentially replaces $T$ with a weakly expanding map on a small interval $[0,\epsilon_0)$. To make this more precise, let $T[x_0,\gamma_n)=[0,x_{n-1})$ and put $\epsilon_0=x_{n-1}$ (note that $\epsilon_0\approx JT(x_0)\,\epsilon_1\asym \epsilon$).
Now define $\Phi^\epsilon:\Delta^\epsilon\rightarrow [0,1]$ by
\begin{displaymath}
\Phi^\epsilon(x,\ell) = \left\{\begin{array}{ll} T^\ell(x)&(x,\ell)\in\Depss\\
T^{\ell+1}(\iota^{-1}(x))&(x,\ell)\in\Deps\end{array}\right..
\end{displaymath}
Note that $\Phi^\epsilon(\Deps_{\ell,n})=[x_{n-\ell},x_{n-\ell-1})$, $\Phi^\epsilon(\Depss_{\ell,i})=[x_{i-\ell},x_{i-1-\ell})$ (for $\ell>0$) and $\Phi^\epsilon:\Deps_0\rightarrow [0,\epsilon_0)$ is $1$--$1$ and given by $T(x+x_0)$.
By defining
\begin{equation}\label{e.defTeps}
T^\epsilon(x)=\left\{\begin{array}{ll} \Phi^\epsilon\circ F^\epsilon\circ(\Phi^\epsilon)^{-1}(x)&x\in[0,x_n),\\
T(x)&x\geq x_n,\end{array}\right.
\end{equation}
$T^\epsilon$ is continuous (in particular $T^\epsilon(x_n)=x_{n-1}=T(x_n)$), $C^2$ near $0$, is a small perturbation of $T$ and satisfies $(T^\epsilon)^\prime(0)=S_0:=\frac{\epsilon_1}{\epsilon_1-\epsilon_2}$. Moreover, $T^\epsilon\circ\Phi^\epsilon=\Phi^\epsilon\circ F^\epsilon$, so $T^\epsilon$ arises as a factor of $F^\epsilon$.

{\em Note:\/} Define $\Phi^*:\Delta\rightarrow [0,1]$ by $\lim_{\epsilon\to 0}\Phi^\epsilon(x)$. This limit makes sense because if $x\in \Delta_{\ell,k}$ then $x\in\Depss$ whenever $\epsilon$ is small enough that $n(\epsilon)>k$. For all such $\epsilon$, $\Phi^\epsilon(x)=T^\ell(x)$. In particular, $\Phi^*|_{\Depss}=\Phi^\epsilon|_{\Depss}$ for all $\epsilon$ and $\Phi^*\circ F = T\circ\Phi^*$.

The eigenfunction $\psi^\epsilon$ from Theorem~\ref{th.4} does not push down to an eigenfunction of the Perron-Frobenius operator for $T$ (or $T^\epsilon$); a small random perturbation is needed. Put
\begin{equation}\label{e.smallpert2}
z_{k+1} = \left\{\begin{array}{ll} \xi_{k+1} &\mbox{if $T(z_k)\in[0,\epsilon_0)$}\\ T(z_k)&\mbox{otherwise},\end{array}\right.
\end{equation}
where the $\xi_k$ are i.i.d. random variables on $[0,\epsilon_0)$ with density function $\rho^\epsilon(z)=\frac{1}{\epsilon_1\,JT(T_{right}^{-1}(z))}$ where $T_{right}^{-1}:[x_0,x_0+\epsilon_1)\to [0,\epsilon_0)$. Notice that $\rho^\epsilon$ is the push-forward of the uniform density on $T_{right}^{-1}[0,\epsilon_0)$, and is close to constant when $\epsilon$ is small (since $T$ is $C^2$ on the right-hand branch). Let $\p_T$ be the Perron-Frobenius operator for $T$ and let $\PTeps$ be the (Markov) transfer operator associated with the process in (\ref{e.smallpert2})\footnote{The small amount of averaging over $[0,\epsilon_0)$ means that $\LTeps$ interpolates between $\p_T$, its Ulam approximations and the two-state metastable model.}.  Let $\Peps:=\p_{\Phi^\epsilon}:L^1(\Delta^\epsilon,\nu^\epsilon)\to L^1([0,1],\leb)$ be the Perron-Frobenius operator for $\Phi^\epsilon$, and introduce an alternate notation for $\Leps$: $\LFeps=\Leps (= \ADeps\circ\PFeps)$.

\begin{lemma}\label{l.semiconjP}
In the notation established above, $\Peps\circ\LFeps=\PTeps\circ\Peps$.
\end{lemma}

\begin{proof}
Let $\PFeps$ be the Perron-Frobenius operator for the system $(F^\epsilon,\Delta^\epsilon)$, let
$\ADeps$ be given by~(\ref{e.piepsilon}) and define $\Aeps :L^1[0,1]\circlearrowleft$ by
\begin{displaymath}
  \Aeps f(z) = \left\{\begin{array}{ll}\rho^\epsilon(z)\,\int_{0}^{\epsilon_0} f\,d(\leb)&z\in[0,\epsilon_0)\\
f(z)&\mbox{otherwise}.\end{array}\right.
\end{displaymath}
Then $\PTeps=\Aeps\circ\p_T$. From~(\ref{e.defTeps}),  $T(x)\in [0,\epsilon_0)$ if and only if $T^\epsilon(x)\in[0,\epsilon_0)$, so in fact $\PTeps=\Aeps\circ\p_T=\Aeps\circ\p_{T^\epsilon}$.

Next, observe that
\begin{displaymath}
\rho^\epsilon = \p_{\Phi^\epsilon}\frac{\chi_{\Deps_0}}{\epsilon_1}=\Peps\frac{\chi_{\Deps_0}}{\epsilon_1}.
\end{displaymath}
Now suppose that $\psi$ is supported on $\Deps_0$. Then, by~(\ref{e.piepsilon}),
\begin{displaymath}
\Peps\circ \ADeps\,\psi = \Peps\chi_{\Deps_0}\,\frac{\int_{\Deps_0}\psi\,d\nu^\epsilon}{\epsilon_1}
=\rho^\epsilon\,\int_{\Depss_0}\psi\,d\nu^\epsilon =\rho^\epsilon\,\int_{[0,\epsilon_0)}\Peps\,\psi\,d(\leb)
=\Aeps\circ\Peps\,\psi.
\end{displaymath}
If $\psi$ is supported on $\Delta^\epsilon\setminus\Deps_0$ then $\ADeps\psi=\psi$ so $\Peps\circ \ADeps\,\psi =\Peps\,\psi$. Moreover, since $\Phi^\epsilon(\Delta^\epsilon\setminus\Deps_0)=(\epsilon_0,1)$, $\Aeps\circ \Peps\,\psi =\Peps\,\psi$. From these two cases
\begin{displaymath}
\Peps\circ \ADeps=\Aeps\circ \Peps
\end{displaymath}
and hence
\begin{displaymath}
\Peps\circ\LFeps=\Peps\circ \ADeps\circ\PFeps=\Aeps\circ \Peps\circ\PFeps
=\Aeps\circ \p_{T^\epsilon}\circ\Peps=\PTeps\circ\Peps
\end{displaymath}
(the third equality is because $\Phi^\epsilon\circ F^\epsilon=T^\epsilon\circ\Phi^\epsilon$).
\end{proof}

\begin{theorem}
\label{th.5}\
\begin{enumerate}
	\item For any $f\in L^1([0,1],\leb)$, $\PTeps f\stackrel{L^1}{\to}\p_Tf$ as $\epsilon\to 0$,
\item $\PTeps$ has an eigenvector $f^{\epsilon}$ satisfying
$\PTeps f^{\epsilon}=f^{\epsilon}$, and $f^{\epsilon}\stackrel{L^1}{\rightarrow} f^*$, where $f^*$ is the density of the unique ACIM for $T$.
\item $\PTeps$ has an eigenvector $h^{\epsilon}$ satisfying $\PTeps\,h^\epsilon=\lambda^\epsilon\,h^{\epsilon}$
(where $\lambda^\epsilon$ is as in Theorem~\ref{th.4}), and $[h^{\epsilon}]^+\stackrel{L^1}{\rightarrow}\frac{1}{2}f^*$
as $\epsilon\to 0$.
\end{enumerate}
\end{theorem}

\begin{proof}\

\emph{1.} First, from the proof of Lemma~\ref{l.semiconjP}, $\PTeps=\Aeps\circ\p_T$ so for any $f\in L^1$,
$$\|\PTeps f-\p_Tf\|_{L^1} = \int_0^{\epsilon_0}|\PTeps f-\p_Tf|\,d(\leb)\leq 2\,\int_0^{\epsilon_0}|\p_Tf|\,d(\leb)\to0$$
as $\epsilon_0\to 0$ since $\p_Tf\in L^1([0,1],\leb)$.

\emph{2.} Next, let $\varphi^\epsilon\in L^1(\Delta^\epsilon,\nu^\epsilon)$ be the (unique, normalised) fixed point of $\LFeps$ and let $f^{\epsilon}=\Peps\varphi^\epsilon$. Let $\varphi^*\in L^1(\Delta,\nu)$ be the (unique) density of the ACIM for $(F,\Delta)$ and let $\Phi^*$ be the canonical semi-conjugacy from $\Delta$ onto $[0,1]$. Then $f^*:=\p_{\Phi^*}\varphi^*$ is the density of the unique ACIM for $T$, and by Lemma~\ref{l.semiconjP}
$f^\epsilon:=\Peps\varphi^\epsilon=\Peps\LFeps\varphi^\epsilon=\PTeps\Peps\varphi^\epsilon=\PTeps f^\epsilon$.
Since $\Phi^*|_{\Depss}=\Phi^\epsilon|_{\Depss}$,
\begin{eqnarray*}
\|f^*-f^{\epsilon}\|_{L^1[0,1]} &=& \left\|  \p_{\Phi^*}(\varphi^*|_{\Delta\setminus\Depss}) -\Peps (\varphi^\epsilon|_{\Deps})
+\Peps((\varphi^*-\varphi^\epsilon)|_{\Depss})\right\|_{L^1[0,1]}\\
&\leq& \|\p_{\Phi^*}(\varphi^*|_{\Delta\setminus\Depss})\|_{L^1[0,1]}
 +  \|\Peps(\varphi^{\epsilon}|_{\Deps})\|_{L^1[0,1]}
 +  \|\Peps(\varphi^*|_{\Depss}-\varphi^{\epsilon}|_{\Depss})\|_{L^1[0,1]}\\
 &=& \| \varphi^*|_{\Delta\setminus\Depss}\|_{L^1(\Delta)}
   + \|\varphi^{\epsilon}|_{\Deps}\|_{L^1(\Delta^\epsilon)}
   + \|\varphi^*|_{\Depss}-\varphi^{\epsilon}|_{\Depss}\|_{L^1(\Delta)}.
   \end{eqnarray*}
   All these terms approach $0$ as $\epsilon\to 0$ by~(\ref{e.towercgc}).

\emph{3.} Finally, we turn to the second eigenfunction.
   Let $g,G$ be such that Theorem~\ref{th.4}
   holds and let $\psi^{\epsilon}\in\cone^{\epsilon,G,g}$ be the fixed point of $\tLeps$.
   Put $h^\epsilon:=\Peps\psi^\epsilon$. Then, by Lemma~\ref{l.semiconjP} and Theorem~\ref{th.4},
   \begin{displaymath}
    \PTeps h^\epsilon = \PTeps\Peps\psi^\epsilon = \Peps\LFeps\psi^\epsilon = \lambda^\epsilon\,\Peps\psi^\epsilon=\lambda^\epsilon\,h^\epsilon.
   \end{displaymath}
     By arguments similar to Theorem~\ref{th.2}, $\psi^\epsilon|_{\Depss}\stackrel{L^1(\Delta)}{\rightarrow}\frac{1}{2}\varphi^*$ as $\epsilon\to 0$.
   Hence
   \begin{equation}\label{e.goodcgc}
   k^\epsilon:=\Peps(\psi^\epsilon|_{\Depss}) \stackrel{L^1[0,1]}{\rightarrow}\frac{1}{2}f^*\qquad\textnormal{as}\qquad\epsilon\to 0.
   \end{equation}
   We need to show that $[h^\epsilon]^+$ has this same limit. This is done by showing that $\|[h^\epsilon]^+-k^\epsilon\|_{L^1[0,1]}\to 0$ as $\epsilon\to0$. To this end, notice that since $\psi^\epsilon|_{\Deps}<0$, we have $h^\epsilon\leq k^\epsilon$ and    in particular $[h^{\epsilon}]^+:=\max\{h^\epsilon,0\}\leq k^\epsilon$.\newline
   {\em Claim:\/} $\lim_{\epsilon\to0}\int_0^1(k_\epsilon-[h_\epsilon]^+)\,d(\leb) =0$.\newline
   {\em Remainder of proof of 3., given claim:\/} Using~(\ref{e.goodcgc}) and the claim, let $\epsilon\to 0$ in
   $$\|f^*/2-[h^{\epsilon}]^+\|_{L^1} \leq \|f^*/2-k^\epsilon\|_{L^1}+\|k^\epsilon-[h^\epsilon]^+\|_{L^1}
   =\|f^*/2-k^\epsilon\|_{L^1} + \int_0^1(k^\epsilon-[h^\epsilon]^+)\,d(\leb).$$
{\em Proof of claim:\/} Fix $\eta=\epsilon^\alpha$ and choose $J$ such that $\eta\in [x_J,x_{J-1}]$. (Notice that $J=\mathcal{O}(\eta^{-\alpha})$.) Since $\psi^\epsilon\in\cone^{\epsilon,G,g}$, (C1.c) gives $\psi^\epsilon|_{\Depss}\leq G\,(1-g/n)^{-n}\leq G^\prime$ for some constant $G^\prime$ independent of $\epsilon$. Now $[0,\eta]\cap\mbox{supp}(k^\epsilon)=(\Phi^\epsilon)^{-1}[0,\eta)\cap \Depss\subset\cup_{j=J-1}^{n-1}(\Phi^\epsilon)^{-1}[x_j,x_{j-1})$ so
  \begin{displaymath}
   \int_0^\eta k^\epsilon\,d(\leb) \leq \sum_{j={J-1}}^{n-1} \int_{(\Phi^\epsilon)^{-1}[x_{j},x_{j-1})}\psi^\epsilon\,d\nu^\epsilon
   \leq G^\prime\,\sum_{\ell=1}^\infty\sum_{j=J}^\infty\nu(\Delta_{\ell,j+\ell})=\mathcal{O}(J^{1-1/\alpha})=\mathcal{O}(\eta^{1-\alpha})
  \end{displaymath}
by~(\ref{e.size}). On the other hand,
\begin{displaymath}
    -\int_\eta^1(h^{\epsilon}-k^\epsilon)\,d(\leb)=\int_{(\Phi^\epsilon)^{-1}[\eta,1)}-\psi^\epsilon|_{\Deps}\,d\nu^\epsilon
    \leq \sum_{\ell=n-J}^n |a_\ell|\,\nu^\epsilon(\Deps_{\ell,n})\leq \mathcal{O}(J\,\epsilon^\alpha)=\mathcal{O}(\eta^{-\alpha}\,\epsilon^\alpha).
\end{displaymath}
    Hence,
    \begin{eqnarray}
    \int_\eta^1 |h^{\epsilon}|\,d(\leb)\geq\int_\eta^1 h^{\epsilon}\,d(\leb) &=& \int_\eta^1 (h^{\epsilon}-k^\epsilon)\,d(\leb)
    +\int_0^1k^\epsilon\,d(\leb)
     - \int_0^\eta k^\epsilon\,d(\leb)\nonumber\\
     &\geq& -\mathcal{O}(\eta^{-\alpha}\,\epsilon^\alpha)+\frac{1}{2}-\mathcal{O}(\eta^{1-\alpha}).\label{e.extraestimate}\end{eqnarray}
     Then, since $\int_0^1h^\epsilon d(\leb)=\int_{\Delta^\epsilon}\psi^\epsilon\,d\nu^\epsilon=0$,
      $\int_0^\eta|h^{\epsilon}|\,d(\leb)\geq\int_0^\eta (-h^{\epsilon})\,d(\leb)=\int_\eta^1h^{\epsilon}\,d(\leb)$ has the same lower bound~(\ref{e.extraestimate}). Since $\eta=\epsilon^{\alpha}$,
      \begin{displaymath}
	      \int_0^1k^\epsilon\,d(\leb)\geq \int_0^1[h^{\epsilon}]^+\,d(\leb)=\frac{1}{2}\int_0^1|h^{\epsilon}|\,d(\leb) \geq \frac{1}{2} - \mathcal{O}(\epsilon^{\alpha\,(1-\alpha)})
	      =\int_0^1k^\epsilon\,d(\leb) - \mathcal{O}(\epsilon^{\alpha\,(1-\alpha)}).
      \end{displaymath}
  \end{proof}


\section{Numerics}\label{sec.numerics}

Ulam's method~\citep{Ulam60} is a well-known (and effective) method for studying~$T$ numerically via its Perron-Frobenius operator. For each $N\in\mathbb{N}$, partition $[0,1]$ into subintervals of length $\frac{1}{N}$ and let~$\mathcal{B}_N$ be the (finite) $\sigma$-algebra obtained by taking unions of elements of the partition; $\mathcal{B}_N$-measurable functions are piecewise constant. Let $\mathbb{E}_N$ denote the conditional expectation operator $\mathbb{E}(\cdot|\mathcal{B}_N)$ acting on function in $L^1(\leb)$. Ulam's method consists in replacing $\p$ with the finite-rank operator $\p_N=\mathbb{E}_N\circ\p$. The leading eigenvalue of $\p_N$ is $1$, and the corresponding eigenvector (fixed point of $\p_N$) is an approximation to a fixed point of~$\p$. Surprisingly (given the absence of a spectral gap for $\p$), these fixed points converge to $\frac{d\mu}{dx}$ as $N\rightarrow\infty$~\citep{Bla08,Mur09}.

Because $\mathbb{E}_N$ is a finite rank projection, each $\p_N$ can be represented by an $N\times N$ matrix~$P_N$. Each $P_N$ is extremely sparse (having $\mathcal{O}(N)$ non-zero entries), and their eigenvalues can be found quickly by iterative methods. Because the dynamics of $T$ are transitive, each $P_N$ is ergodic, so the eigenvalue~$1$ has strictly larger modulus than the other eigenvalues. Interestingly, the {\em spectral gap\/} (separation of the largest modulus eigenvalue~$1$ from the next largest modulus eigenvalue) scales as $N^{-\alpha}$.

\subsection{Eigenvalue scaling}
The two-state model of Section \ref{sec.metastable} showed that when the geometric escape rate from the set $[0,\epsilon_0]$ approached zero more slowly than the escape rate from the set $[\epsilon_0,1]$, the gap from $1$ of the second eigenvalue of the two-state Markov chain scaled like the slower escape rate from $[0,\epsilon_0]$;  namely $\epsilon_0^\alpha$.

We now replace the two-state model with the ``$N$-state model'' $P_N$ arising from Ulam's method. The matrix $P_N$ is row-stochastic, representing the transitions of a finite state Markov chain whose $i$th state is identified with the subinterval $J_i:=[(i-1)/N,i/N)$. The indifferent fixed point at 0 can
be associated naturally with the subinterval $J_1=[0,1/N)\approx[0,\epsilon_0)$. The conditional transition probabilities out of state~$1$ are $(P_N)_{11}=\leb(J_1\cap T^{-1}J_1)/\leb(J_1)= 1-N\,T^{-1}(1/N)$ and $(P_N)_{12}=1-(P_N)_{11}$.  Thus the geometric rate of escape from $J_1$ is $-\log(P_N)_{11}\approx 1-(P_N)_{11} \asym N^{-\alpha}\asym \epsilon_0^\alpha$, and this is of the same order as previously computed for the two-state model.

We find numerically that despite increasing the number of states from two to $N$, the second eigenvalue of our $N$-state Ulam matrix retains the scaling predicted by the two-state model when $\epsilon_0=1/N$, namely $1-\lambda_2(N)\sim N^{-\alpha}$;  see Figure \ref{fig.2}.

\begin{figure}[htb]
\begin{center}
	\subfigure[]{\label{fig.2a}\includegraphics[height=6.5cm]{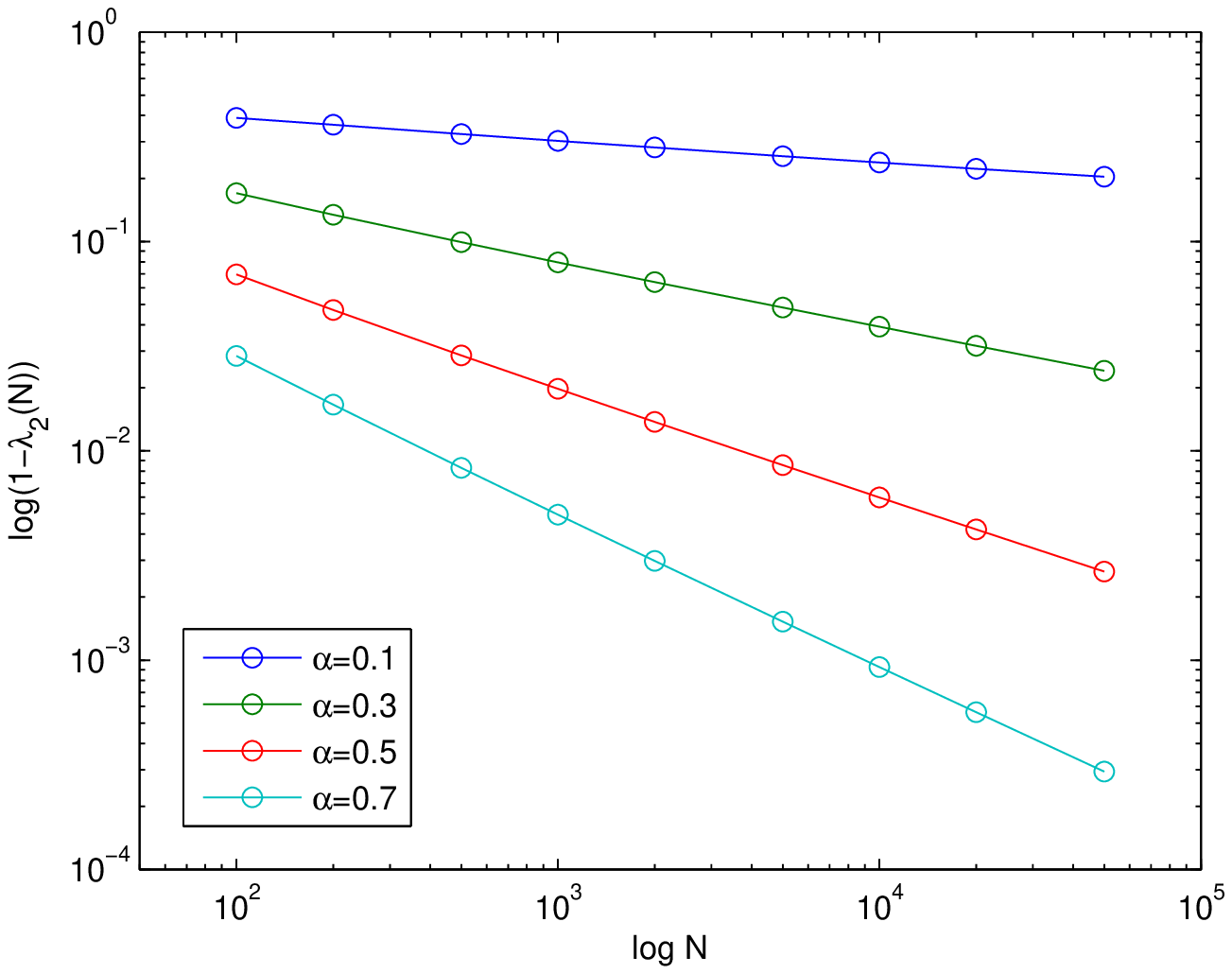}}\subfigure[]{\label{fig.2b}\includegraphics[height=6.5cm]{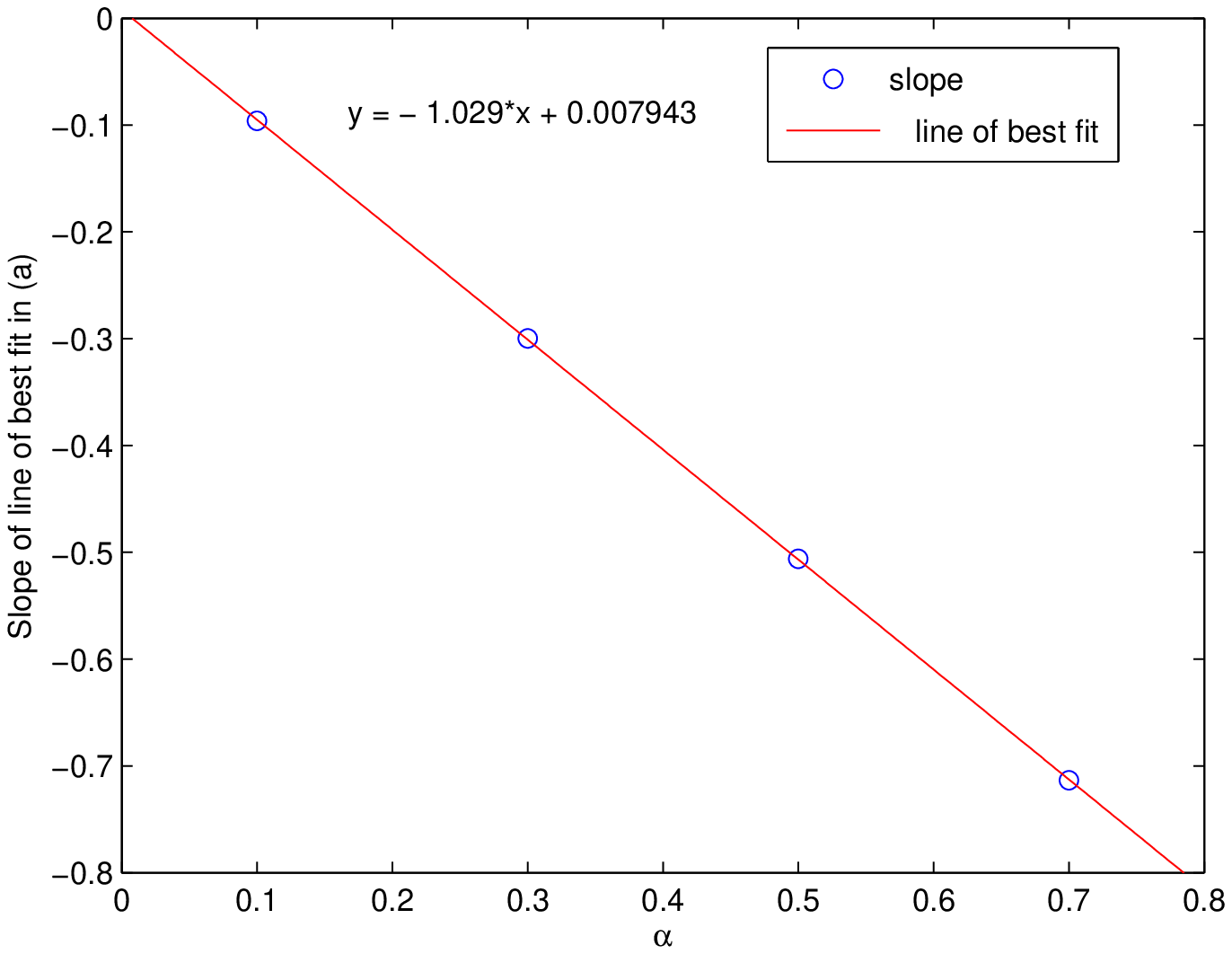}}
  \caption{(a) Variation of the second eigenvalue of $P_N$ with $N$ and $\alpha$. (b) Slope of line of best fit for each $\alpha$. Note: Computed by {\tt eigs\/} in {\sc Matlab}, with Ulam matrices for the PM map~\citep[Example~3]{Mur09}.}\label{fig.2}
\end{center}
\end{figure}

\paragraph{Connection with the Perron-Frobenius operators $\mathcal{L}^\epsilon_{F^\epsilon}$ and $\mathcal{L}^\epsilon_T$.}
Theorems ~\ref{th.4} and~\ref{th.5} prove the existence of a second eigenvalue\footnote{More precisely, we show that there is another real eigenvalue very close to 1;  based on numerical computations we conjecture that the eigenvalue $\lambda^\epsilon$ is indeed the second largest real eigenvalue.} $\lambda^\epsilon$.  The matrix $P_N$ successfully reproduces the dynamics responsible for this eigenvalue, and we now explicitly describe the connection. Set $\epsilon=1/(N\,T^\prime(x_0^+))$ and choose $n,\epsilon_1,\epsilon_2$ as in (\ref{e.size})--(\ref{e.small}). Then $T[x_0,x_0+\epsilon_1)=:[0,\epsilon_0)\approx [0,1/N)=J_1$. The averaging created by replacing $T$ by $T^\epsilon$ (see~(\ref{e.defTeps})) or the small random perturbation~(\ref{e.smallpert2}) ``almost linearises'' $T$ over $[0,1/N]$, so that $P_N$ is a very near approximation\footnote{Since $T$ is locally $C^2$ and $\epsilon_0\approx\frac{1}{N}$, the transition probabilities $(P_N)_{11}$ and $(P_N)_{12}$ are close to the corresponding entries of the matrix for $\Peps\circ \mathcal{L}_T^{\epsilon_0}$ (note that $\Peps$ --- like the perturbation defined by~(\ref{e.smallpert2}) --- {\em averages over a small interval $[0,\mathcal{O}(\epsilon))$\/}. Outside $[0,\epsilon_0]$, $T$ is uniformly expanding (albeit rather weakly) and $\mathcal{L}_T^{\epsilon_0}=\mathcal{P}_T$.} to~$\mathcal{L}_T^{\epsilon_0}$. Theorem~\ref{th.4} predicts an eigenvalue of $\mathcal{L}_T^{\epsilon_0}$, $\lambda^{\epsilon_0}\in(1-2\frac{\epsilon_2}{\epsilon_1},1-\frac{\epsilon_2}{\epsilon_1})$. Numerical computations with $P_N$ for a range of $N$ produce second eigenvalues within this range. In fact, the upper limit is a very good estimate; see Table~\ref{table.1}.

\begin{table}[htb]
\begin{center}
\begin{tabular}{|l|l|l|}
  \hline
  $N$&$1-\lambda_2(N)$&$\epsilon_2/\epsilon_1$\\
  \hline\hline
  100 & 0.069494728128226
 &  0.060750416292176\\
  200 & 0.047118990434159 & 0.042626262679704 \\
  500 & 0.028582682402957 & 0.026696029895732
 \\
  1000 & 0.019751285772241 & 0.018706181316717 \\
  2000 & 0.013727390048589 &  0.013165183357731\\
  5000 & 0.008542396305559
 & 0.008301674655368 \\
  10000 & 0.005988977377968
 & 0.005866565930472 \\
  20000 & 0.004208535921532 & 0.004150111773511 \\
  50000 & 0.002646628586393 & 0.002621525600809
 \\
  \hline
\end{tabular}
\caption{Comparison of $1-\lambda_2(N)$ computed numerically as a second eigenvalue of the $N\times N$ Ulam matrix and the corresponding lower bound $\epsilon_2/\epsilon_1$ obtained from Theorems~\ref{th.4} and~\ref{th.5} ($\alpha=0.5$).}\label{table.1}
\end{center}
\end{table}

\subsection{Ulam's method and the escape rate from $[1/N,1]$}

We conclude with some simple remarks about how to observe the ACCIMs $\mu_n$, and their geometric escape rates, numerically. The open tower systems constructed in Section~\ref{sec.firsteigenfunctions} exclude a set $H_n$ corresponding to those orbits which land in $[0,x_n)$ under~$T$. Letting $\Phi^*:\Delta\to [0,1]$ be defined by $\Phi^*(x,\ell)=T^\ell(x)$, the measure $\mu\circ(\Phi^*)^{-1}$ is an ACIM for $T$, and $\mu_n\circ(\Phi^*)^{-1}$ is an ACCIM for $T|_{[x_n,1]}$ with geometric escape rate~$1-\mu_n(H^1_n)$ (see Corollary~\ref{cor.2}). In fact, $\mu_n(H^1_n)\asym\epsilon$, where $\epsilon\asym x_n$ ($x_n$ is the $n$th left preimage of the discontinuity point $x_0$ for $T$). Hence, if $x_n\approx\frac{1}{N}$, then one expects the escape rate from $[x_n,1]$ to scale like $1/N$. Note that with choosing $\epsilon\asym 1/N$, the corresponding $n$ in the constructions of Sections~\ref{sec.firsteigenfunctions} and~\ref{sec.secondeigenfunctions} is $n\asym \epsilon^{-\alpha}\asym N^{\alpha}$. Now partition the $N\times N$ Ulam matrix~$P_N$ as
\begin{displaymath}
P_N=\left[\begin{array}{c|c} (P_N)_{11}&\mathbf{a}^T\\ \hline \mathbf{b}&P^o_N\end{array}\right]
\end{displaymath}
where $\mathbf{a},\mathbf{b}$ are $(N-1)$-vectors and $P^o_N$ is an $(N-1)\times(N-1)$ matrix. In fact, $P^o_N$ is the Ulam approximation to the conditional Perron-Frobenius operator $\chi_{[1/N,1]}\p(\,\cdot\,\chi_{[1/N,1]})$. In Figure~\ref{fig.3} we present numerical evidence that the leading eigenvalue $\lambda^o_1(N)$ for $P^o_N$ has the scaling $1-\lambda^o_1(N) \asym 1/N$, {\em independently of $\alpha$\/}.

\begin{figure}[htb]
\begin{center}
  \includegraphics[width=12cm]{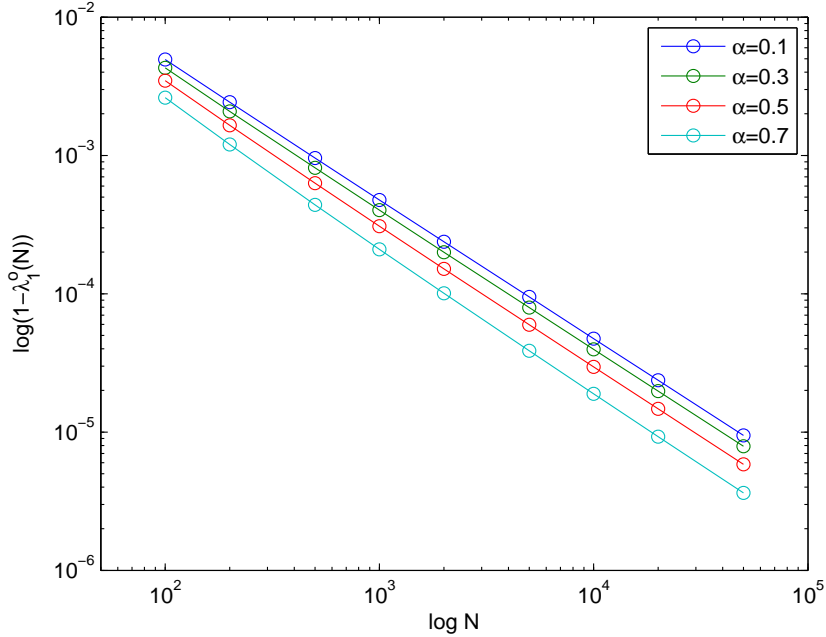}\\
  \caption{Variation of escape rate from $[1/N,1]$ with $N$ and $\alpha$.}\label{fig.3}
\end{center}
\end{figure}

Finally, Theorem~\ref{th.2} predicts the convergence of the ACCIMs $\mu_n$ to the ACCIM as the size of the hole~$H_n$ shrinks to $0$. We illustrate this convergence numerically as follows. For a large $N_*$ (we have used $N_*=10^5$), form $P_{N_*}$ and calculate the leading eigenvector. This is a good approximation to the density of the ACIM for $T$~\citep{Mur09}, and we use it as a reference measure. Next, for a sequence of smaller $N_k$ (we used the values from the first column of Table~\ref{table.1}), calculate the leading eigenvector of $P^o_{N_k}$. Comparing the probability measure induced by these eigenvectors with the reference measure from the Ulam aproximation $P_{N_*}$ we see good convergence in Figure~\ref{fig.4}.

\begin{figure}[htb]
\begin{center}
  \includegraphics[width=12cm]{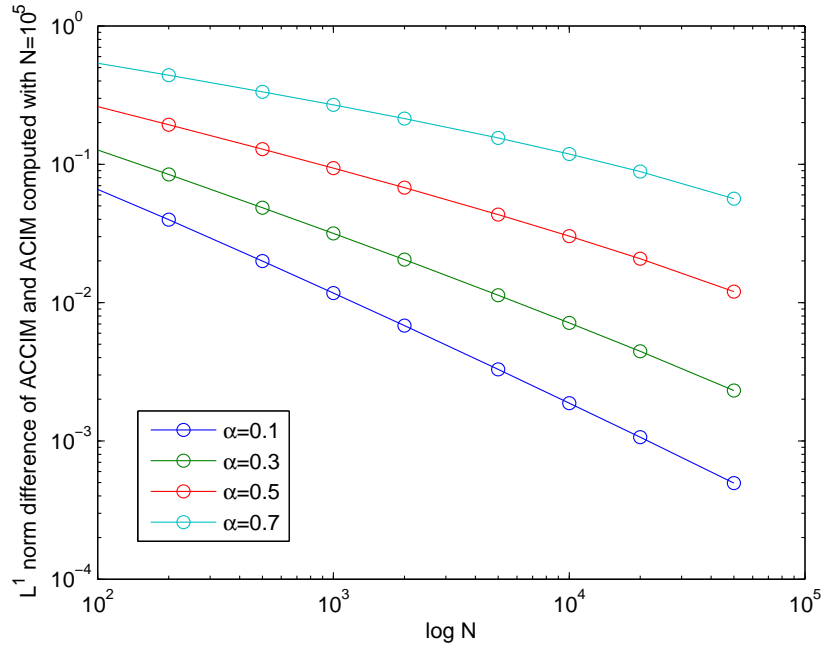}\\
  \caption{Total variation norm error between Ulam approximate ACCIMs with a hole $[0,1/N]$ and the Ulam approximate ACIM with a bin size of $10^{-5}$. A range of different PM maps are used, ($\alpha$ is order of tangency of the indifferent fixed point).}\label{fig.4}
\end{center}
\end{figure}

\section*{Acknowledgements}
GF and OS thank the University of Canterbury for its hospitality during a visit that initiated this work. GF is partially supported by the ARC Discovery Project DP110100068 and OS is supported by the ARC Centre of Excellence for Mathematics and Statistics of Complex Systems (MASCOS).
RM thanks the University of New South Wales for hospitality, allowing part of this work to be completed.


\end{document}